\documentclass{tac}

\usepackage{amsmath,amsfonts,mathabx,tikz,tikz-cd,enumitem,stmaryrd,url,mathtools}

\usetikzlibrary{matrix,arrows,decorations.pathmorphing}
\usepackage[colorlinks=true]{hyperref}
\hypersetup{allcolors=[rgb]{0.1,0.1,0.4}}

\author{Branko Nikoli\'c and Ross Street}

\thanks{The first author gratefully acknowledges the support of an International Macquarie University Research Scholarship.
The second author gratefully acknowledges the support of Australian Research Council Discovery Grant DP160101519.}

\address{Mathematics Department, Macquarie University, \\ NSW 2109, Australia}

\title {Hopf rings for grading and differentials}

\copyrightyear{2018}

\eaddress{branko.nikolic@mq.edu.au, ross.street@mq.edu.au}

\theoremstyle{definition}
\newtheorem{defn}{Definition}[section]
\newtheorem{prop}{Proposition}[section]

\newtheorem{lemma}{Lemma}[section]
\newtheorem{example}{Example}[section]

\begin{document}

\maketitle
\begin{abstract}
In the category of abelian groups, Pareigis constructed a Hopf ring whose comodules are differential graded abelian groups. We show that this Hopf ring can be obtained by combining grading and differential Hopf rings using semidirect product in fairly general symmetric monoidal additive categories.
\end{abstract}

\section{Introduction}

Chain complexes of abelian groups (also called differential graded abelian groups) form a symmetric monoidal closed category DGAb (explained in detail in Section \ref{sec:DGAb}) which can be obtained as a category of comodules for a Hopf ring\footnote{The word ``ring'' will denote a monoid in an additive monoidal category. We save ``(co)algebra'' for (co)algebras for a (co)monad. We use ``(co)module''
for coalgebras for (co)monads obtained by tensoring with a (co)monoid.} $P$ in the symmetric monoidal closed category Ab of abelian groups \cite{Pareigis1981}. DGAb has a full symmetric monoidal closed subcategory GAb, consisting of graded abelian groups seen as complexes with zero differential, which can also be obtained as a category of comodules of a different Hopf ring in Ab.

The main new insight, our first construction, occurs in a braided monoidal additive category $\mathcal{W}$, begins with an object $D$ for which the braiding $\sigma_{D,D} : D\otimes D\to D\otimes D$ is minus the identity morphism of
$D\otimes D$, and produces a Hopf ring structure on the object $H=I\oplus D$
(where $I$ is the unit for tensor $\otimes$). 

The second construction imports the additive group of integers $\mathbb{Z}$ 
into a symmetric monoidal additive category $\mathcal{V}$ by tensoring with
the unit to obtain a Hopf ring $Z=\mathbb{Z}\cdot I$ equipped with a braiding
coelement $\gamma : Z\otimes Z\to I$.  

The third construction, called bosonization in \cite{Majid2006}, is a generalisation of the semidirect product $M\rtimes G$ of a group $G$ with a $G$-module $M$.
At a very general level \cite{Brug2011} it concerns the question of 
(co)monadicity of a composite of two (co)monadic functors.  
We work in any symmetric monoidal category $\mathcal{V}$ (additivity is not needed). 
Given any (Hopf) bimonoid $A$ equipped with a braiding coelement $\gamma$ in $\mathcal{V}$, the category $\mathrm{Comod}_{\mathcal{V}}(A)$ of
$A$-comodules becomes braided monoidal (as explained for example in \cite{Joyal1991} for
the case where $\mathcal{V}$ is vector spaces). 
So the concept of Hopf
monoid $H$ makes sense in $\mathrm{Comod}_{\mathcal{V}}(A)$. 
The semidirect product $H\rtimes A$ is a Hopf monoid in $\mathcal{V}$ for which 
$$\mathrm{Comod}_{\mathcal{V}}(H\rtimes A) \simeq 
\mathrm{Comod}_{\mathrm{Comod}_{\mathcal{V}}(A)}(H) \ .$$ 

In Sections~\ref{sec:DGAb} and \ref{sec:GAb} we review the monoidal categories of differential graded and graded abelian groups. Section~\ref{sec:semi} explains the
semidirect product construction. The heart of the paper is Section~\ref{sec:birings}
where we see the first and second constructions. 
It culminates by showing that the Pareigis Hopf ring $P$ is an example of
the third construction $H\rtimes A$
where $A$ is obtained by the second construction and $H$ the first.          
          
\section{Differential graded abelian groups}\label{sec:DGAb}

The category DGAb (differential graded abelian groups) has chain complexes $A$ as objects. They are defined by diagrams
\begin{equation}\label{eq:diffArrows}
\ldots\xrightarrow{d}A_{n+1}\xrightarrow{d}A_{n}\xrightarrow{d}A_{n-1}\xrightarrow{d}\ldots
\end{equation}
in Ab with group homomorphisms $d=d^A_n : A_n\to A_{n-1}$ satisfying $d\circ d=0$. 
An arrow $f:A\rightarrow B$, called a chain map, consists of group homomorphisms
$f_n:A_n\rightarrow B_n$, indexed by integers, satisfying
\begin{equation}\label{eq:arrDGAb}
f_n\circ d=d\circ f_{n-1}\,.
\end{equation}

DGAb is monoidal with tensor product defined by
\begin{eqnarray}\label{eq:monDGAb}
\begin{aligned}
(A\otimes B)_n = \sum_{i+j=n} A_i\otimes B_j \phantom{AAAAAAAAA}\\
d(a\otimes b) = da\otimes b + (-1)^i a\otimes db,\;\text{for}\;a\in A_i\;\text{and}\;b\in B_j\,.\end{aligned}
\end{eqnarray}
The unit is given by $I_n=\delta_{n0}\mathbb{Z}$ (Kronecker delta). 
There is a symmetry
\begin{align}
\sigma(a\otimes b)=(-1)^{ij}b\otimes a
\end{align}
and a closed structure
\begin{eqnarray}\label{eq:clDGAb}
\begin{aligned}
[B,C]_n=\prod_{j}\mathrm{Ab}(B_j,C_{j+n}) \phantom{AAAAAAAAA}\\
(df)_j b=d(f_j(b))-(-1)^nf_{j-1}(db),\;\text{for}\;f\in [B,C]_n\;\text{and}\;b\in B_j\,.
\end{aligned}
\end{eqnarray}

The monoidal category DGAb (with direction of arrows in (\ref{eq:diffArrows}) inverted) can be obtained (see \cite{Pareigis1981}) as the category of Eilenberg-Moore (EM) coalgebras for the monoidal comonad obtained by tensoring with the Hopf ring $P$ defined by
\begin{align}\label{PareigisEx}
P&=\mathbb{Z}\langle {\xi,\xi^{-1},\psi} \rangle/(\xi\psi+\psi\xi,\psi^2)\\
\Delta(\xi)&= \xi\otimes \xi,\; 
\epsilon(\xi)=1,\; s(\xi)=\xi^{-1} \nonumber \\
\Delta(\psi)&= \psi\otimes 1+\xi^{-1}\otimes \psi,\;
\epsilon(\psi)=0,\; s(\psi)= \psi\xi \nonumber
\end{align}
where $\Delta$ is the comultiplication, $s$ is the antipode, and corner brackets denote non-commutativity.

We will see in the end how this ring can be slightly modified to give the direction of arrows as in (\ref{eq:diffArrows}). It amounts to exchanging $\xi$ and $\xi^{-1}.$

\section{Graded abelian groups}\label{sec:GAb}

The category GAb of graded abelian groups can be seen as a full subcategory of DGAb consisting of chain complexes with all $d=0$. 
GAb inherits a symmetric monoidal closed structure, which follows from (\ref{eq:monDGAb}) and (\ref{eq:clDGAb}). 
On the other hand, there is a forgetful functor $U:\mathrm{DGAb}\rightarrow \mathrm{GAb}$, with adjoints $L\dashv U\dashv R$
given by
\begin{align}
%\mathrm{GAb}&\xrightarrow \mathrm{DGAb}
L(C)_n&=C_{n+1}\oplus C_n\\
R(C)_n&=C_{n}\oplus C_{n-1}\\
d&=\left[\begin{array}{cc}
    0 & 1	\\
    0 & 0 
    \end{array}\right]\,.
\end{align}
$U$ reflects isomorphisms, since $f^{-1}$ satisfies (\ref{eq:arrDGAb}) if and only if $f$ does. The functor $U$, having both adjoints, preserves all limits and colimits, in particular $U\text{-}$split equalizers and coequalizers. Hence, $U$ is both comonadic and monadic (see Chapter 3, Theorem 10 of \cite{Barr1985} where the word ``triple'' is
used for ``monad'').
%There is an isomoprhism of comonads on GAb
%\begin{equation}
%UR\cong (SI\oplus I)\otimes -
%\end{equation}

There is a functor $\Sigma:\mathrm{GAb}\rightarrow \mathrm{Ab}$ that takes the coproduct (sum) of all components. It has a right adjoint which creates $\mathbb{Z}$ copies of each abelian group. The diagram below summarises all relevant adjunctions.
\begin{equation}\label{diag:adjs}
\begin{tikzpicture}[scale=2, every node/.style={scale=1},baseline=(current  bounding  box.center)]
% % % left hand side
%objects
\node (d) at (-2,0) {$\mathrm{DGAb}$};
\node (g) at (0,0) {$\mathrm{GAb}$};
\node (a) at (2,0) {$\mathrm{Ab}$};
\node (sym1) at (-1,0.2) {$\bot$};
\node (sym2) at (-1,-0.2) {$\bot$};
\node (sym3) at (1,-0.2) {$\bot$};
%arrows
\path[,->,font=\scriptsize,>=angle 90,bend right]
(g) edge node[above] {$L$} (d);
\path[,->,font=\scriptsize,>=angle 90]
(d) edge node[below left] {$U$} (g);
\path[->,font=\scriptsize,>=angle 90,bend left]
(g)edge node[below] {$R$} (d);
\path[,->,font=\scriptsize,>=angle 90,bend left]
(a) edge node[below] {$C$} (g);
\path[,->,font=\scriptsize,>=angle 90]
(g) edge node[above] {$\Sigma$} (a);
%\path[<-,font=\scriptsize,>=angle 90,bend right,out=45,in=-45]
%(A2)edge node[right] {$\mathrm{Caten}(\mathcal{X},\mathcal{G})$} (A2);
%
\end{tikzpicture}
\end{equation}

Both $\Sigma \circ C$ and $U\circ R$ are comonads isomorphic to tensoring with a certain Hopf ring in Ab and GAb respectively. In Section \ref{sec:semi} we discuss the semidirect product construction in general, and then in section \ref{sec:birings} show that the Pareigis biring is the semidirect product of the two birings generating $\Sigma \circ C$ and  $U\circ R$.

\section{Semidirect product} \label{sec:semi}

Let $\mathcal{W}= \mathrm{Comod}_{\mathcal{V}}(A)$ be the category of 
comodules for a (Hopf) bimonoid $A$ in a symmetric monoidal $\mathcal{V}$.
For $\mathcal{W}$ to be braided we need $A$ to have a braiding coelement 
$\gamma: A\otimes A \rightarrow I$ satisfying the duals of the three axioms at page 58 of \cite{Joyal1991}, which we quote here in the form we are going to use later (we read the string diagrams from bottom to top):
\begin{equation}\label{ax:elt1}
\begin{tikzpicture}[baseline=(current  bounding  box.center)]
\def \strx {0.5}
\def \stry {0.2}
\def \angmnd {0}
\node (lhs) at (-1,0) {\begin{tikzpicture}
\coordinate (inA) at (-1 * \strx, 5 * \stry);
\coordinate (elt) at (1 * \strx, 3 * \stry);
\coordinate (cmndA) at (-1 * \strx, 3 * \stry);
\coordinate (mndA1) at (-0.5 * \strx, -1 * \stry);
\coordinate (mndA2) at (-1 * \strx, -3 * \stry);
\coordinate (outA1) at (-1 * \strx, -5 * \stry);
\coordinate (outA2) at (1 * \strx, -5 * \stry);
\draw[very thin] (inA) to[out=-90,in=90] (cmndA);
\draw[very thin] (elt) to[out=180,in=\angmnd] (mndA1);
\draw[very thin] (elt) to[out=0,in=\angmnd] (mndA2);
\draw[very thin] (cmndA) to[out=180+\angmnd,in=180-\angmnd] (mndA2);
\draw[very thin] (cmndA) to[out=-\angmnd,in=180-\angmnd] (mndA1);
\draw[very thin] (mndA1) to[out=-90,in=90] (outA2);
\draw[very thin] (mndA2) to[out=-90,in=90] (outA1);
\end{tikzpicture}};
\node (eq) at (0,0) {=};
\node (rhs) at (1,0) {\begin{tikzpicture}
\coordinate (inA) at (1 * \strx, 5 * \stry);
\coordinate (elt) at (-1 * \strx, 3 * \stry);
\coordinate (cmndA) at (1 * \strx, 3 * \stry);
\coordinate (mndA1) at (-1 * \strx, -1 * \stry);
\coordinate (mndA2) at (1 * \strx, -1 * \stry);
\coordinate (outA1) at (-1 * \strx, -5 * \stry);
\coordinate (outA2) at (1 * \strx, -5 * \stry);
\draw[very thin] (inA) to[out=-90,in=90] (cmndA);
\draw[very thin] (elt) to[out=180,in=180-\angmnd] (mndA1);
\draw[very thin] (elt) to[out=0,in=180-\angmnd] (mndA2);
\draw[very thin] (cmndA) to[out=180+\angmnd,in=\angmnd] (mndA1);
\draw[very thin] (cmndA) to[out=-\angmnd,in=\angmnd] (mndA2);
\draw[very thin] (mndA1) to[out=-90,in=90] (outA2);
\draw[very thin] (mndA2) to[out=-90,in=90] (outA1);
\end{tikzpicture}};
\end{tikzpicture}
\end{equation}
\begin{equation}\label{ax:elt2}
\begin{tikzpicture}[baseline=(current  bounding  box.center)]
\def \strx {0.5}
\def \stry {0.8}
\def \angmnd {0}
\node (lhs) at (-1,0) {\begin{tikzpicture}
\coordinate (elt) at (0 * \strx, 1 * \stry);
\coordinate (cmndA) at (1 * \strx, 0.5 * \stry);
\coordinate (outA1) at (-1 * \strx, -1 * \stry);
\coordinate (outA2) at (0.5 * \strx, -1 * \stry);
\coordinate (outA3) at (1 * \strx, -1 * \stry);
\draw[very thin] (elt) to[out=0,in=90] (cmndA);
\draw[very thin] (elt) to[out=180,in=90] (outA1);
\draw[very thin] (cmndA) to[out=180,in=90] (outA2);
\draw[very thin] (cmndA) to[out=0,in=90] (outA3);
\end{tikzpicture}};
\node (eq) at (0,0) {=};
\node (rhs) at (1,0) {\begin{tikzpicture}
\coordinate (elt1) at (0 * \strx, 1 * \stry);
\coordinate (elt2) at (0 * \strx, 0.5 * \stry);
\coordinate (mndA) at (-1 * \strx, -0.5 * \stry);
\coordinate (outA1) at (-1 * \strx, -1 * \stry);
\coordinate (outA2) at (0.5 * \strx, -1 * \stry);
\coordinate (outA3) at (1 * \strx, -1 * \stry);
\draw[very thin] (elt1) to[out=180,in=180-\angmnd] (mndA);
\draw[very thin] (elt1) to[out=0,in=90] (outA3);
\draw[very thin] (elt2) to[out=180,in=\angmnd] (mndA);
\draw[very thin] (elt2) to[out=0,in=90] (outA2);
\draw[very thin] (mndA) to[out=-90,in=90] (outA1);
\end{tikzpicture}};
\end{tikzpicture}
\end{equation}
\begin{equation}\label{ax:elt3}
\begin{tikzpicture}[baseline=(current  bounding  box.center)]
\def \strx {0.5}
\def \stry {0.8}
\def \angmnd {0}
\node (lhs) at (-1,0) {\begin{tikzpicture}
\coordinate (elt1) at (-1 * \strx, 1 * \stry);
\coordinate (elt2) at (1 * \strx, 1 * \stry);
\coordinate (mndA) at (0.5 * \strx, -0.5 * \stry);
\coordinate (outA1) at (-1.5 * \strx, -1 * \stry);
\coordinate (outA2) at (-0.5 * \strx, -1 * \stry);
\coordinate (outA3) at (0.5 * \strx, -1 * \stry);
\draw[very thin] (elt1) to[out=0,in=180-\angmnd] (mndA);
\draw[very thin] (elt1) to[out=180,in=90] (outA1);
\draw[very thin] (elt2) to[out=0,in=\angmnd] (mndA);
\draw[very thin] (elt2) to[out=180,in=90] (outA2);
\draw[very thin] (mndA) to[out=-90,in=90] (outA3);
\end{tikzpicture}};
\node (eq) at (0,0) {=};
\node (rhs) at (1,0) {\begin{tikzpicture}
\coordinate (elt) at (0 * \strx, 1 * \stry);
\coordinate (cmndA) at (-1 * \strx, 0.5 * \stry);
\coordinate (outA1) at (-1 * \strx, -1 * \stry);
\coordinate (outA2) at (-0.5 * \strx, -1 * \stry);
\coordinate (outA3) at (1 * \strx, -1 * \stry);
\draw[very thin] (elt) to[out=180,in=90] (cmndA);
\draw[very thin] (elt) to[out=0,in=90] (outA3);
\draw[very thin] (cmndA) to[out=180,in=90] (outA1);
\draw[very thin] (cmndA) to[out=0,in=90] (outA2);
\end{tikzpicture}};
\end{tikzpicture}
\end{equation}
Explicitly, the braiding of $(X,\alpha_X)$ and $(Y,\alpha_Y)$ is\footnote{We often omit writing $\otimes$.}
\begin{equation*}
\sigma_{(X,\alpha_X),(Y,\alpha_Y)}=XY\xrightarrow{\sigma_{X,Y}}YX\xrightarrow{\alpha_Y\alpha_X}AYAX\xrightarrow{1\sigma_{Y,A} 1}AAYX\xrightarrow{\gamma 11}YX
\end{equation*}
%\begin{equation}
%s_{XY}=XY\xrightarrow{\gamma 11}AAXY\xrightarrow{1\sigma 1}AXAY\xrightarrow{\alpha_X\alpha_Y}XY\xrightarrow{\sigma}YX
%\end{equation}

For an $A$-comodule $(X,\alpha_X)$ define
%\begin{equation}
%\tau_X=AX\xrightarrow{\delta 1} AAX \xrightarrow{1\sigma_{AX}}AXA
%\xrightarrow{\alpha_X 1} XA\,.
%\end{equation}
\begin{equation}
\tau_X=XA\xrightarrow{\alpha_X 1} AXA \xrightarrow{1\sigma}AAX
\xrightarrow{\mu 1} AX\,.
\end{equation}
\begin{prop}
If $X$ is a comonoid in $\mathcal{W}$, then $\tau_X$ is a distributive law in $\mathcal{V}$ (or equivalently, of the comonad $X\otimes -$ over the comonad $A\otimes -$).
\end{prop}
\begin{proof}
There are four axioms to check. The two involving counit and comultiplication for $A$ use the compatibility of counit with the multiplication of $A$, and the bimonoid axiom, respectively.
The two involving counit and comultiplication for $X$ follow from the fact that they are $A$-comodule morphisms.
\end{proof}

Let $H$ be a bimonoid in $\mathcal{W}$. 
It automatically inherits a (co)monoid structure in $\mathcal{V}$ by forgetting that (co)unit and (co)multiplication maps are $A$-comodule morphisms. 
Note that, unless $\gamma=\epsilon\otimes\epsilon$, 
$H$ need not be a bimonoid in $\mathcal{V}$. 
\begin{defn}
The semidirect product $H\rtimes A$ of a bimonoid $A$ in $\mathcal{V}$ and a bimonoid $(H,\alpha_H)$ in $\mathcal{W}=\mathrm{Comod}_\mathcal{V}(A)$ 
is the object $H\otimes A$, with comonoid structure given via the distributive law 
$\tau_H$, and monoid structure via the distributive law $s_{AH}$. 
Using a thick line for $H$, thin line for $A$, we depict the comultiplication 
and the multiplication of $H\rtimes A$ by the following string diagrams
\begin{equation}
\begin{tikzpicture}[baseline=(current  bounding  box.center)]
\def \strx {0.5}
\def \stry {0.2}
\def \angmnd {0}
\def \angio {10}

\coordinate (inH1) at (-3 * \strx, 6 * \stry);
\coordinate (inA1) at (-1 * \strx, 6 * \stry);
\coordinate (inH2) at (1 * \strx, 6 * \stry);
\coordinate (inA2) at (3 * \strx, 6 * \stry);

\coordinate (mndH) at (-1 * \strx, -3 * \stry);
\coordinate (cmndA1) at (-0.5 * \strx, 2 * \stry);
\coordinate (act1) at (-0.5 * \strx, 0 * \stry);
\coordinate (mndA1) at (1 * \strx, -3 * \stry);
\coordinate (cmndA2) at (1 * \strx, -6 * \stry);
\coordinate (cmndH) at (-1 * \strx, -6 * \stry);

\draw[thick] (inH1) to[out=-90+\angio,in=180-\angmnd] (mndH);
\draw[very thin] (inA1) to[out=-90+\angio,in=90] (cmndA1);
\draw[thick] (inH2) to[out=-90-\angio,in=\angmnd] (act1);
\draw[very thin] (inA2) to[out=-90-\angio,in=\angmnd] (mndA1);
% % % % % % % % % % % %
\draw[very thin] (cmndA1) to[out=180+\angmnd,in=180-\angmnd] (act1);
\draw[very thin] (cmndA1) to[out=-\angmnd,in=180-\angmnd] (mndA1);
% % % % % % % % % % % % % % % %
\draw[thick] (act1) to[out=-90,in=\angmnd] (mndH);
% % % % % % % % % % % % % % % %
\draw[thick] (mndH) to[out=-90,in=90] (cmndH);
\draw[very thin] (mndA1) to[out=-90,in=90] (cmndA2);
\end{tikzpicture}
\,\,\,\,\,\,
\begin{tikzpicture}[baseline=(current  bounding  box.center)]
\def \strx {0.5}
\def \stry {0.2}
\def \angmnd {0}
\def \angio {10}

\coordinate (mndH) at (-1 * \strx, 6 * \stry);
\coordinate (mndA1) at (1 * \strx, 6 * \stry);

\coordinate (cmndA2) at (1 * \strx, 3 * \stry);
\coordinate (cmndH) at (-1 * \strx, 3 * \stry);
\coordinate (elt) at (-1 * \strx, 1 * \stry);
\coordinate (act2) at (-0.5 * \strx, -2 * \stry);
\coordinate (mndA2) at (0.5 * \strx, -2 * \stry);

\coordinate (outH1) at (-3 * \strx, -6 * \stry);
\coordinate (outA1) at (-1 * \strx, -6 * \stry);
\coordinate (outH2) at (1 * \strx, -6 * \stry);
\coordinate (outA2) at (3 * \strx, -6 * \stry);

\draw[thick] (mndH) to[out=-90,in=90] (cmndH);
\draw[very thin] (mndA1) to[out=-90,in=90] (cmndA2);
% % % % % % % % % % % % % % % % % % % %
\draw[thick] (cmndH) to[out=180+\angmnd,in=90-\angio] (outH1);
\draw[thick] (cmndH) to[out=-\angmnd,in=\angmnd] (act2);
\draw[very thin] (cmndA2) to[out=180+\angmnd,in=\angmnd] (mndA2);
\draw[very thin] (cmndA2) to[out=-\angmnd,in=90+\angio] (outA2);
% % % % % % % % % % % % % % % % % % % % %
\draw[very thin] (elt) to[out=180,in=180-\angmnd] (act2);
\draw[very thin] (elt) to[out=0,in=180-\angmnd] (mndA2);
% % % % % % % % % % % % % % % % % % % % %
\draw[thick] (act2) to[out=-90,in=90+\angio] (outH2);
\draw[very thin] (mndA2) to[out=-90,in=90-\angio] (outA1);
\end{tikzpicture}
\end{equation}
where all relevant arrows in $\mathcal{V}$ are uniquely determined by their source and target, so there is no need for labelling.
\end{defn}
\begin{prop}\label{prop:semi}
The semidirect product $H\rtimes A$ is a bimonoid in $\mathcal{V}$. If $H$ and $A$ are Hopf, with antipodes graphically represented by dots, then so is $H\rtimes A$, with the antipode given by diagram (\ref{ap}).
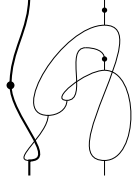
\begin{equation}\label{ap}
\begin{tikzpicture}[baseline=(current  bounding  box.center)]
\def \strx {0.5}
\def \stry {0.2}
\def \angmnd {0}
\def \angio {0}

\coordinate (inH) at (-1 * \strx, 6 * \stry);
\coordinate (inA) at (1 * \strx, 6 * \stry);

\coordinate (apH) at (-1.5 * \strx, 0 * \stry);
\fill (apH) circle (1.5pt);
\coordinate (apA1) at (1 * \strx, 5 * \stry);
\fill (apA1) circle (1pt);

\coordinate (cmndA1) at (1 * \strx, 4 * \stry);
\coordinate (elt) at (0.5 * \strx, 2.5 * \stry);
\coordinate (apA2) at (1 * \strx, 1.75 * \stry);
\fill (apA2) circle (1pt);
\coordinate (cmndA2) at (1 * \strx, 1 * \stry);
\coordinate (mndA1) at (-0 * \strx, -1 * \stry);
\coordinate (mndA2) at (-0.5 * \strx, -2 * \stry);
\coordinate (mndA3) at (1 * \strx, -5 * \stry);

\coordinate (act) at (-1 * \strx, -5 * \stry);

\coordinate (outH) at (-1 * \strx, -6 * \stry);
\coordinate (outA) at (1 * \strx, -6 * \stry);

\draw[thick] (inH) to[out=-90+\angio,in=90] (apH)
				   to[out=-90,in=\angmnd] (act)
				   to[out=-90,in=90-\angio] (outH);
\draw[very thin] (inA) to[out=-90-\angio,in=90] (apA1)
					   to[out=-90,in=90] (cmndA1);
\draw[very thin] (cmndA1) to[out=180+\angmnd,in=180-\angmnd] (mndA2);
\draw[very thin] (cmndA1) to[out=-\angmnd,in=180-\angmnd] (mndA3)
						  to[out=-90,in=90+\angio] (outA);
\draw[very thin] (elt) to[out=180,in=\angmnd] (mndA1);
\draw[very thin] (elt) to[out=0,in=90](apA2)
					   to[out=-90,in=90] (cmndA2);
\draw[very thin] (cmndA2) to[out=180+\angmnd,in=180-\angmnd] (mndA1);
\draw[very thin] (cmndA2) to[out=-\angmnd,in=\angmnd] (mndA3);
\draw[very thin] (mndA1) to[out=-90,in=\angmnd] (mndA2)
						 to[out=-90,in=180-\angmnd] (act);
\end{tikzpicture}
\end{equation}
\end{prop}
\begin{proof}
\begin{figure}\label{fig:diagramsSemi}
\def \strx {0.25 * 1.45}
\def \stry {0.1 * 1.45}
\def \angmnd {0}
\def \angio {10}
\begin{align}
%1
\begin{tikzpicture}[baseline=(current  bounding  box.center)]
\def \angcross{30};
\coordinate (crossH) at (-0.25 * \strx, 0 * \stry);
\coordinate (crossA) at (0.25 * \strx, 0 * \stry);
\coordinate (inH1) at (-3 * \strx, 10 * \stry);
\coordinate (inA1) at (-1 * \strx, 10 * \stry);
\coordinate (inH2) at (1 * \strx, 10 * \stry);
\coordinate (inA2) at (3 * \strx, 10 * \stry);
\coordinate (cmndH1) at (-2.5 * \strx, 8 * \stry);
\coordinate (cmndA1) at (-1.5 * \strx, 8 * \stry);
\coordinate (cmndH2) at (1.5 * \strx, 8 * \stry);
\coordinate (cmndA2) at (2.5 * \strx, 8 * \stry);
\coordinate (elt1) at (-2.75 * \strx, 6 * \stry);
\coordinate (elt12) at (1.5 * \strx, 6 * \stry);
\coordinate (mndA4) at (-1.5 * \strx, 4 * \stry);
\coordinate (mndA6) at (2.5 * \strx, 4 * \stry);
\coordinate (act21) at (-2.5 * \strx, 4 * \stry);
\coordinate (act4) at (1.5 * \strx, 4 * \stry);
\coordinate (cmndA3) at (-2.0 * \strx, -4 * \stry);
\coordinate (cmndA11) at (1.5 * \strx, -4 * \stry);
\coordinate (act3) at (-2.0 * \strx, -7 * \stry);
\coordinate (act22) at (1.5 * \strx, -7 * \stry);
\coordinate (mndH1) at (-2.5 * \strx, -8 * \stry);
\coordinate (mndA2) at (-1 * \strx, -8 * \stry);
\coordinate (mndH2) at (1 * \strx, -8 * \stry);
\coordinate (mndA1) at (2.5 * \strx, -8 * \stry);
\coordinate (outH1) at (-3 * \strx, -10 * \stry);
\coordinate (outA1) at (-1 * \strx, -10 * \stry);
\coordinate (outH2) at (1 * \strx, -10 * \stry);
\coordinate (outA2) at (3 * \strx, -10 * \stry);
\draw[very thin] (cmndA2) to[out=-\angmnd,in=90] (3 * \strx,0)
to [out=-90,in=\angmnd] (mndA1);
\draw[very thin] (cmndA1) to[out=-\angmnd,in=180-\angcross] (crossA);
\draw[very thin] (crossA) to[out=-\angcross,in=90] (cmndA11);
\draw[very thin] (inA2) to[out=-90-\angio,in=90] (cmndA2);
\draw[very thin] (mndA1) to[out=-90,in=90+\angio] (outA2);
\draw[very thin] (cmndA2) to[out=180+\angmnd,in=\angmnd] (mndA6);
\draw[thick] (inH1) to[out=-90+\angio,in=90] (cmndH1);
\draw[very thin] (inA1) to[out=-90+\angio,in=90] (cmndA1);
\draw[thick] (inH2) to[out=-90-\angio,in=90] (cmndH2);
% % % % % % % % % % % %
\draw[very thin] (cmndA1) to[out=180+\angmnd,in=\angmnd] (mndA4);
\draw[very thin] (cmndA11) to[out=-\angmnd,in=180-\angmnd] (mndA1);
\draw[very thin] (cmndA11) to[out=180+\angmnd,in=180-\angmnd] (act22);
%\draw[thick] (act1) to[out=-90,in=90] (cmndH2);
% % % % % % % % % % % % % % % %
\draw[thick, tension=1] (cmndH1) to[out=180+\angmnd,in=90](-3 * \strx,0) to[out=-90,in=180-\angmnd] (mndH1);
\draw[thick] (cmndH1) to[out=-\angmnd,in=\angmnd] (act21);
\draw[thick] (cmndH2) to[out=180+\angmnd,in=\angcross] (crossH);
\draw[thick] (crossH) to[out=180+\angcross,in=\angmnd] (act3);
\draw[thick] (cmndH2) to[out=-\angmnd,in=\angmnd] (act4);
\draw[very thin] (mndA4) to[out=-90,in=90] (cmndA3);
\draw[very thin] (cmndA3) to[out=180+\angmnd,in=180-\angmnd] (act3);
\draw[very thin] (cmndA3) to[out=-\angmnd,in=180-\angmnd] (mndA2);
\draw[thick] (act3) to[out=-90,in=\angmnd] (mndH1);
% % % % % % % % % % % % % % % % % % % %
\draw[thick] (mndH1) to[out=-90,in=90-\angio] (outH1);
% % % % % % % % % % % % % % % % % % % % %
\draw[very thin] (elt1) to[out=180,in=180-\angmnd] (act21);
\draw[very thin] (elt12) to[out=0,in=180-\angmnd] (mndA6);
\draw[thick] (act21) to[out=-90,in=180-\angcross] (crossH);
\draw[thick] (crossH) to[out=-\angcross,in=180-\angmnd] (mndH2);
\draw[thick] (act22) to[out=-90,in=\angmnd] (mndH2);
\draw[very thin] (elt1) to[out=0,in=180-\angmnd] (mndA4);
\draw[very thin] (elt12) to[out=180,in=180-\angmnd] (act4);
\draw[thick] (act4) to[out=-90,in=\angmnd] (act22);
\draw[very thin] (mndA6) to[out=-90,in=\angcross] (crossA);
\draw[very thin] (crossA) to[out=180+\angcross,in=\angmnd] (mndA2);
% % % % % % % % % % % % % % % % % % % % %
\draw[thick] (mndH2) to[out=-90,in=90+\angio] (outH2);
\draw[very thin] (mndA2) to[out=-90,in=90-\angio] (outA1);
\end{tikzpicture}
&=
%2
\begin{tikzpicture}[baseline=(current  bounding  box.center)]
\coordinate (inH1) at (-3 * \strx, 10 * \stry);
\coordinate (inA1) at (-1 * \strx, 10 * \stry);
\coordinate (inH2) at (1 * \strx, 10 * \stry);
\coordinate (inA2) at (3 * \strx, 10 * \stry);
\coordinate (cmndA1) at (-0.5 * \strx, 8 * \stry);
\coordinate (cmndA11) at (1 * \strx, 6 * \stry);
\coordinate (cmndH1) at (-2.5 * \strx, 8 * \stry);
\coordinate (cmndH2) at (1 * \strx, 8 * \stry);
\coordinate (act3) at (-2.5 * \strx, -0.5 * \stry);
\coordinate (mndA4) at (-2.5 * \strx, 3 * \stry);
\coordinate (mndH1) at (-2.5 * \strx, -8 * \stry);
\coordinate (mndH2) at (0.5 * \strx, -8 * \stry);
\coordinate (elt1) at (-3 * \strx, 6 * \stry);
\coordinate (cmndA3) at (-2.5 * \strx, 1 * \stry);
\coordinate (elt12) at (0.5 * \strx, 3 * \stry);
\coordinate (mndA6) at (1 * \strx, 1 * \stry);
\coordinate (act21) at (-3.5 * \strx, 2.5 * \stry);
\coordinate (act22) at (1.5 * \strx, -7 * \stry);
\coordinate (act4) at (2 * \strx, -4 * \stry);
\coordinate (mndA2) at (-0.5 * \strx, -7 * \stry);
\coordinate (outH1) at (-3 * \strx, -10 * \stry);
\coordinate (outA1) at (-1 * \strx, -10 * \stry);
\coordinate (outH2) at (1 * \strx, -10 * \stry);
\coordinate (outA2) at (3 * \strx, -10 * \stry);
\coordinate (mndA1) at (3 * \strx, -5 * \stry);
\coordinate (cmndA2) at (3 * \strx, 5 * \stry);
\draw[very thin] (cmndA2) to[out=-\angmnd,in=\angmnd] (mndA1);
\draw[very thin] (cmndA1) to[out=-\angmnd,in=90] (cmndA11);
\draw[very thin] (inA2) to[out=-90-\angio,in=90] (cmndA2);
\draw[very thin] (mndA1) to[out=-90,in=90+\angio] (outA2);
\draw[very thin] (cmndA2) to[out=180+\angmnd,in=\angmnd] (mndA6);
\draw[thick] (inH1) to[out=-90+\angio,in=90] (cmndH1);
\draw[very thin] (inA1) to[out=-90+\angio,in=90] (cmndA1);
\draw[thick] (inH2) to[out=-90-\angio,in=90] (cmndH2);
% % % % % % % % % % % %
\draw[very thin] (cmndA1) to[out=180+\angmnd,in=\angmnd] (mndA4);
\draw[very thin] (cmndA11) to[out=-\angmnd,in=180-\angmnd] (mndA1);
\draw[very thin] (cmndA11) to[out=180+\angmnd,in=180-\angmnd] (act22);
%\draw[thick] (act1) to[out=-90,in=90] (cmndH2);
% % % % % % % % % % % % % % % %
\draw[thick] (cmndH1) to[out=180+\angmnd,in=180-\angmnd] (mndH1);
\draw[thick] (cmndH1) to[out=-\angmnd,in=\angmnd] (act21);
\draw[thick] (cmndH2) to[out=180+\angmnd,in=\angmnd] (act3);
\draw[thick] (cmndH2) to[out=-\angmnd,in=\angmnd] (act4);
\draw[very thin] (mndA4) to[out=-90,in=90] (cmndA3);
\draw[very thin] (cmndA3) to[out=180+\angmnd,in=180-\angmnd] (act3);
\draw[very thin] (cmndA3) to[out=-\angmnd,in=180-\angmnd] (mndA2);
\draw[thick] (act3) to[out=-90,in=\angmnd] (mndH1);
% % % % % % % % % % % % % % % % % % % %
\draw[thick] (mndH1) to[out=-90,in=90-\angio] (outH1);
% % % % % % % % % % % % % % % % % % % % %
\draw[very thin] (elt1) to[out=180,in=180-\angmnd] (act21);
\draw[very thin] (elt12) to[out=0,in=180-\angmnd] (mndA6);
\draw[thick] (act21) to[out=-90,in=180-\angmnd] (mndH2);
\draw[thick] (act22) to[out=-90,in=\angmnd] (mndH2);
\draw[very thin] (elt1) to[out=0,in=180-\angmnd] (mndA4);
\draw[very thin] (elt12) to[out=180,in=180-\angmnd] (act4);
\draw[thick] (act4) to[out=-90,in=\angmnd] (act22);
\draw[very thin] (mndA6) to[out=-90,in=\angmnd] (mndA2);
% % % % % % % % % % % % % % % % % % % % %
\draw[thick] (mndH2) to[out=-90,in=90+\angio] (outH2);
\draw[very thin] (mndA2) to[out=-90,in=90-\angio] (outA1);
\end{tikzpicture}
&=
%3
\begin{tikzpicture}[baseline=(current  bounding  box.center)]
\coordinate (inH1) at (-3 * \strx, 10 * \stry);
\coordinate (inA1) at (-1 * \strx, 10 * \stry);
\coordinate (inH2) at (1 * \strx, 10 * \stry);
\coordinate (inA2) at (3 * \strx, 10 * \stry);
\coordinate (cmndA1) at (-0.5 * \strx, 8 * \stry);
\coordinate (cmndA11) at (1 * \strx, 6 * \stry);
\coordinate (cmndH1) at (-2.5 * \strx, 8 * \stry);
\coordinate (cmndH2) at (1 * \strx, 8 * \stry);
\coordinate (act3) at (-2.5 * \strx, -0.5 * \stry);
\coordinate (mndA4) at (-2.5 * \strx, 3 * \stry);
\coordinate (mndH1) at (-2.5 * \strx, -8 * \stry);
\coordinate (mndH2) at (0.5 * \strx, -8 * \stry);
\coordinate (elt1) at (-3 * \strx, 6 * \stry);
\coordinate (cmndA3) at (-2.5 * \strx, 1 * \stry);
\coordinate (elt12) at (0.5 * \strx, 3 * \stry);
\coordinate (mndA6) at (1 * \strx, 1 * \stry);
\coordinate (mndA3) at (0 * \strx, -3 * \stry);
\coordinate (act21) at (-3.5 * \strx, 2.5 * \stry);
\coordinate (act22) at (1.5 * \strx, -7 * \stry);
\coordinate (mndA2) at (-0.5 * \strx, -7 * \stry);
\coordinate (outH1) at (-3 * \strx, -10 * \stry);
\coordinate (outA1) at (-1 * \strx, -10 * \stry);
\coordinate (outH2) at (1 * \strx, -10 * \stry);
\coordinate (outA2) at (3 * \strx, -10 * \stry);
\coordinate (mndA1) at (3 * \strx, -5 * \stry);
\coordinate (cmndA2) at (3 * \strx, 5 * \stry);
\draw[very thin] (cmndA2) to[out=-\angmnd,in=\angmnd] (mndA1);
\draw[very thin] (cmndA1) to[out=-\angmnd,in=90] (cmndA11);
\draw[very thin] (inA2) to[out=-90-\angio,in=90] (cmndA2);
\draw[very thin] (mndA1) to[out=-90,in=90+\angio] (outA2);
\draw[very thin] (cmndA2) to[out=180+\angmnd,in=\angmnd] (mndA6);
\draw[thick] (inH1) to[out=-90+\angio,in=90] (cmndH1);
\draw[very thin] (inA1) to[out=-90+\angio,in=90] (cmndA1);
\draw[thick] (inH2) to[out=-90-\angio,in=90] (cmndH2);
% % % % % % % % % % % %
\draw[very thin] (cmndA1) to[out=180+\angmnd,in=\angmnd] (mndA4);
\draw[very thin] (cmndA11) to[out=-\angmnd,in=180-\angmnd] (mndA1);
\draw[very thin] (cmndA11) to[out=180+\angmnd,in=180-\angmnd] (mndA3);
%\draw[thick] (act1) to[out=-90,in=90] (cmndH2);
% % % % % % % % % % % % % % % %
\draw[thick] (cmndH1) to[out=180+\angmnd,in=180-\angmnd] (mndH1);
\draw[thick] (cmndH1) to[out=-\angmnd,in=\angmnd] (act21);
\draw[thick] (cmndH2) to[out=180+\angmnd,in=\angmnd] (act3);
\draw[thick] (cmndH2) to[out=-\angmnd,in=\angmnd] (act22);
\draw[very thin] (mndA4) to[out=-90,in=90] (cmndA3);
\draw[very thin] (cmndA3) to[out=180+\angmnd,in=180-\angmnd] (act3);
\draw[very thin] (cmndA3) to[out=-\angmnd,in=180-\angmnd] (mndA2);
\draw[thick] (act3) to[out=-90,in=\angmnd] (mndH1);
% % % % % % % % % % % % % % % % % % % %
\draw[thick] (mndH1) to[out=-90,in=90-\angio] (outH1);
% % % % % % % % % % % % % % % % % % % % %
\draw[very thin] (elt1) to[out=180,in=180-\angmnd] (act21);
\draw[very thin] (elt12) to[out=0,in=180-\angmnd] (mndA6);
\draw[thick] (act21) to[out=-90,in=180-\angmnd] (mndH2);
\draw[thick] (act22) to[out=-90,in=\angmnd] (mndH2);
\draw[very thin] (elt1) to[out=0,in=180-\angmnd] (mndA4);
\draw[very thin] (elt12) to[out=180,in=\angmnd] (mndA3);
\draw[very thin] (mndA3) to[out=-90,in=180-\angmnd] (act22);
\draw[very thin] (mndA6) to[out=-90,in=\angmnd] (mndA2);
% % % % % % % % % % % % % % % % % % % % %
\draw[thick] (mndH2) to[out=-90,in=90+\angio] (outH2);
\draw[very thin] (mndA2) to[out=-90,in=90-\angio] (outA1);
\end{tikzpicture}
\label{d1}\\
&=
%4
\begin{tikzpicture}[baseline=(current  bounding  box.center)]
\coordinate (inH1) at (-3 * \strx, 10 * \stry);
\coordinate (inA1) at (-1 * \strx, 10 * \stry);
\coordinate (inH2) at (1 * \strx, 10 * \stry);
\coordinate (inA2) at (3 * \strx, 10 * \stry);
\coordinate (cmndA1) at (-0.5 * \strx, 8 * \stry);
\coordinate (cmndA11) at (-1 * \strx, 7 * \stry);
\coordinate (cmndA4) at (-0.5 * \strx, 2 * \stry);
\coordinate (cmndH1) at (-2.5 * \strx, 8 * \stry);
\coordinate (cmndH2) at (1 * \strx, 8 * \stry);
\coordinate (act3) at (-2.5 * \strx, -0.5 * \stry);
\coordinate (mndA4) at (-2.5 * \strx, 1 * \stry);
\coordinate (mndH1) at (-2.5 * \strx, -8 * \stry);
\coordinate (mndH2) at (0.5 * \strx, -8 * \stry);
\coordinate (elt1) at (-3 * \strx, 6 * \stry);
\coordinate (cmndA3) at (-2.5 * \strx, 3 * \stry);
\coordinate (elt12) at (0.5 * \strx, 3 * \stry);
\coordinate (mndA6) at (-1 * \strx, -4 * \stry);
\coordinate (mndA3) at (0 * \strx, 0 * \stry);
\coordinate (act21) at (-3.5 * \strx, 2.5 * \stry);
\coordinate (act22) at (1.5 * \strx, -7 * \stry);
\coordinate (mndA2) at (-0.5 * \strx, -7 * \stry);
\coordinate (mndA5) at (-0.5 * \strx, -2 * \stry);
\coordinate (outH1) at (-3 * \strx, -10 * \stry);
\coordinate (outA1) at (-1 * \strx, -10 * \stry);
\coordinate (outH2) at (1 * \strx, -10 * \stry);
\coordinate (outA2) at (3 * \strx, -10 * \stry);
\coordinate (mndA1) at (3 * \strx, -5 * \stry);
\coordinate (cmndA2) at (3 * \strx, 5 * \stry);
\draw[very thin] (cmndA2) to[out=-\angmnd,in=\angmnd] (mndA1);
\draw[very thin] (cmndA1) to[out=-\angmnd,in=180-\angmnd] (mndA1);
\draw[very thin] (inA2) to[out=-90-\angio,in=90] (cmndA2);
\draw[very thin] (mndA1) to[out=-90,in=90+\angio] (outA2);
\draw[very thin] (cmndA2) to[out=180+\angmnd,in=\angmnd] (mndA2);
\draw[thick] (inH1) to[out=-90+\angio,in=90] (cmndH1);
\draw[very thin] (inA1) to[out=-90+\angio,in=90] (cmndA1);
\draw[thick] (inH2) to[out=-90-\angio,in=90] (cmndH2);
% % % % % % % % % % % %
\draw[very thin] (cmndA1) to[out=180+\angmnd,in=90] (cmndA11);
\draw[very thin] (cmndA11) to[out=-\angmnd,in=90] (cmndA4);
\draw[very thin] (cmndA4) to[out=180+\angmnd,in=180-\angmnd] (mndA5);
\draw[very thin] (cmndA4) to[out=-\angmnd,in=180-\angmnd] (mndA3);
\draw[very thin] (cmndA11) to[out=180+\angmnd,in=\angmnd] (mndA4);
%\draw[thick] (act1) to[out=-90,in=90] (cmndH2);
% % % % % % % % % % % % % % % %
\draw[thick] (cmndH1) to[out=180+\angmnd,in=180-\angmnd] (mndH1);
\draw[thick] (cmndH1) to[out=-\angmnd,in=\angmnd] (act21);
\draw[thick] (cmndH2) to[out=180+\angmnd,in=\angmnd] (act3);
\draw[thick] (cmndH2) to[out=-\angmnd,in=\angmnd] (act22);
\draw[very thin] (mndA4) to[out=-90,in=180-\angmnd] (act3);
\draw[very thin] (cmndA3) to[out=180+\angmnd,in=180-\angmnd] (mndA4);
\draw[very thin] (cmndA3) to[out=-\angmnd,in=180-\angmnd] (mndA6);
\draw[thick] (act3) to[out=-90,in=\angmnd] (mndH1);
% % % % % % % % % % % % % % % % % % % %
\draw[thick] (mndH1) to[out=-90,in=90-\angio] (outH1);
% % % % % % % % % % % % % % % % % % % % %
\draw[very thin] (elt1) to[out=180,in=180-\angmnd] (act21);
\draw[very thin] (elt12) to[out=0,in=\angmnd] (mndA5);
\draw[very thin] (mndA5) to[out=-90,in=\angmnd] (mndA6);
\draw[thick] (act21) to[out=-90,in=180-\angmnd] (mndH2);
\draw[thick] (act22) to[out=-90,in=\angmnd] (mndH2);
\draw[very thin] (elt1) to[out=0,in=90] (cmndA3);
\draw[very thin] (elt12) to[out=180,in=\angmnd] (mndA3);
\draw[very thin] (mndA3) to[out=-90,in=180-\angmnd] (act22);
\draw[very thin] (mndA6) to[out=-90,in=180-\angmnd] (mndA2);
% % % % % % % % % % % % % % % % % % % % %
\draw[thick] (mndH2) to[out=-90,in=90+\angio] (outH2);
\draw[very thin] (mndA2) to[out=-90,in=90-\angio] (outA1);
\end{tikzpicture}
&=
%5
\begin{tikzpicture}[baseline=(current  bounding  box.center)]
\coordinate (inH1) at (-3 * \strx, 10 * \stry);
\coordinate (inA1) at (-1 * \strx, 10 * \stry);
\coordinate (inH2) at (1 * \strx, 10 * \stry);
\coordinate (inA2) at (3 * \strx, 10 * \stry);
\coordinate (cmndA1) at (-0.5 * \strx, 8 * \stry);
\coordinate (cmndA11) at (-1 * \strx, 7 * \stry);
\coordinate (cmndA4) at (0.5 * \strx, 3 * \stry);
\coordinate (cmndH1) at (-2.5 * \strx, 8 * \stry);
\coordinate (cmndH2) at (1 * \strx, 8 * \stry);
\coordinate (act3) at (-2.5 * \strx, -0.5 * \stry);
\coordinate (mndA4) at (-2.5 * \strx, 1 * \stry);
\coordinate (mndH1) at (-2.5 * \strx, -8 * \stry);
\coordinate (mndH2) at (0.5 * \strx, -8 * \stry);
\coordinate (elt1) at (-3 * \strx, 6 * \stry);
\coordinate (cmndA3) at (-2.5 * \strx, 3 * \stry);
\coordinate (elt12) at (-0.5 * \strx, 2 * \stry);
\coordinate (mndA6) at (-1 * \strx, -4 * \stry);
\coordinate (mndA3) at (0.5 * \strx, -1.5 * \stry);
\coordinate (act21) at (-3.5 * \strx, 2.5 * \stry);
\coordinate (act22) at (1.5 * \strx, -7 * \stry);
\coordinate (mndA2) at (-0.5 * \strx, -7 * \stry);
\coordinate (mndA5) at (-1 * \strx, -1.5 * \stry);
\coordinate (outH1) at (-3 * \strx, -10 * \stry);
\coordinate (outA1) at (-1 * \strx, -10 * \stry);
\coordinate (outH2) at (1 * \strx, -10 * \stry);
\coordinate (outA2) at (3 * \strx, -10 * \stry);
\coordinate (mndA1) at (3 * \strx, -5 * \stry);
\coordinate (cmndA2) at (3 * \strx, 5 * \stry);
\draw[very thin] (cmndA2) to[out=-\angmnd,in=\angmnd] (mndA1);
\draw[very thin] (cmndA1) to[out=-\angmnd,in=180-\angmnd] (mndA1);
\draw[very thin] (inA2) to[out=-90-\angio,in=90] (cmndA2);
\draw[very thin] (mndA1) to[out=-90,in=90+\angio] (outA2);
\draw[very thin] (cmndA2) to[out=180+\angmnd,in=\angmnd] (mndA2);
\draw[thick] (inH1) to[out=-90+\angio,in=90] (cmndH1);
\draw[very thin] (inA1) to[out=-90+\angio,in=90] (cmndA1);
\draw[thick] (inH2) to[out=-90-\angio,in=90] (cmndH2);
% % % % % % % % % % % %
\draw[very thin] (cmndA1) to[out=180+\angmnd,in=90] (cmndA11);
\draw[very thin] (cmndA11) to[out=-\angmnd,in=90] (cmndA4);
\draw[very thin] (cmndA4) to[out=180+\angmnd,in=\angmnd] (mndA5);
\draw[very thin] (cmndA4) to[out=-\angmnd,in=\angmnd] (mndA3);
\draw[very thin] (cmndA11) to[out=180+\angmnd,in=\angmnd] (mndA4);
%\draw[thick] (act1) to[out=-90,in=90] (cmndH2);
% % % % % % % % % % % % % % % %
\draw[thick] (cmndH1) to[out=180+\angmnd,in=180-\angmnd] (mndH1);
\draw[thick] (cmndH1) to[out=-\angmnd,in=\angmnd] (act21);
\draw[thick] (cmndH2) to[out=180+\angmnd,in=\angmnd] (act3);
\draw[thick] (cmndH2) to[out=-\angmnd,in=\angmnd] (act22);
\draw[very thin] (mndA4) to[out=-90,in=180-\angmnd] (act3);
\draw[very thin] (cmndA3) to[out=180+\angmnd,in=180-\angmnd] (mndA4);
\draw[very thin] (cmndA3) to[out=-\angmnd,in=180-\angmnd] (mndA6);
\draw[thick] (act3) to[out=-90,in=\angmnd] (mndH1);
% % % % % % % % % % % % % % % % % % % %
\draw[thick] (mndH1) to[out=-90,in=90-\angio] (outH1);
% % % % % % % % % % % % % % % % % % % % %
\draw[very thin] (elt1) to[out=180,in=180-\angmnd] (act21);
\draw[very thin] (elt12) to[out=180,in=180-\angmnd] (mndA5);
\draw[very thin] (mndA5) to[out=-90,in=180-\angmnd] (act22);
\draw[thick] (act21) to[out=-90,in=180-\angmnd] (mndH2);
\draw[thick] (act22) to[out=-90,in=\angmnd] (mndH2);
\draw[very thin] (elt1) to[out=0,in=90] (cmndA3);
\draw[very thin] (elt12) to[out=0,in=180-\angmnd] (mndA3);
\draw[very thin] (mndA3) to[out=-90,in=\angmnd] (mndA6);
\draw[very thin] (mndA6) to[out=-90,in=180-\angmnd] (mndA2);
% % % % % % % % % % % % % % % % % % % % %
\draw[thick] (mndH2) to[out=-90,in=90+\angio] (outH2);
\draw[very thin] (mndA2) to[out=-90,in=90-\angio] (outA1);
\end{tikzpicture}
\label{d2}\\
&=
%6
\begin{tikzpicture}[baseline=(current  bounding  box.center)]
\coordinate (inH1) at (-3 * \strx, 10 * \stry);
\coordinate (inA1) at (-1 * \strx, 10 * \stry);
\coordinate (inH2) at (1 * \strx, 10 * \stry);
\coordinate (inA2) at (3 * \strx, 10 * \stry);
\coordinate (cmndA1) at (-0.5 * \strx, 8 * \stry);
\coordinate (cmndA11) at (-1 * \strx, 7 * \stry);
\coordinate (act11) at (-0.75 * \strx, 4 * \stry);
\coordinate (act12) at (1 * \strx, 2 * \stry);
\coordinate (cmndH1) at (-2.5 * \strx, 8 * \stry);
\coordinate (cmndH2) at (1 * \strx, 8 * \stry);
\coordinate (act3) at (-2.5 * \strx, -3 * \stry);
\coordinate (mndA4) at (-2.5 * \strx, 1 * \stry);
\coordinate (mndH1) at (-2.5 * \strx, -8 * \stry);
\coordinate (mndH2) at (0.5 * \strx, -8 * \stry);
\coordinate (elt1) at (-3 * \strx, 6 * \stry);
\coordinate (cmndA3) at (-2 * \strx, 3 * \stry);
\coordinate (elt12) at (0 * \strx, 2 * \stry);
\coordinate (mndA6) at (-1 * \strx, -2 * \stry);
\coordinate (act21) at (-3.5 * \strx, 2 * \stry);
\coordinate (act22) at (1.5 * \strx, -7 * \stry);
\coordinate (mndA2) at (-0.5 * \strx, -7 * \stry);
\coordinate (mndA5) at (0 * \strx, -4.5 * \stry);
\coordinate (outH1) at (-3 * \strx, -10 * \stry);
\coordinate (outA1) at (-1 * \strx, -10 * \stry);
\coordinate (outH2) at (1 * \strx, -10 * \stry);
\coordinate (outA2) at (3 * \strx, -10 * \stry);
\coordinate (mndA1) at (2 * \strx, 1 * \stry);
\coordinate (cmndA2) at (2 * \strx, -3 * \stry);
\draw[very thin] (inA2) to[out=-90-\angio,in=\angmnd] (mndA1);
\draw[very thin] (cmndA1) to[out=-\angmnd,in=180-\angmnd] (mndA1);
\draw[very thin] (mndA1) to[out=-90,in=90] (cmndA2);
\draw[very thin] (cmndA2) to[out=-\angmnd,in=90+\angio] (outA2);
\draw[very thin] (cmndA2) to[out=180+\angmnd,in=\angmnd] (mndA2);
\draw[thick] (inH1) to[out=-90+\angio,in=90] (cmndH1);
\draw[very thin] (inA1) to[out=-90+\angio,in=90] (cmndA1);
\draw[thick] (inH2) to[out=-90-\angio,in=90] (cmndH2);
% % % % % % % % % % % %
\draw[very thin] (cmndA1) to[out=180+\angmnd,in=90] (cmndA11);
\draw[very thin] (cmndA11) to[out=-\angmnd,in=\angmnd] (mndA5);
\draw[very thin] (cmndA11) to[out=180+\angmnd,in=\angmnd] (mndA4);
%\draw[thick] (act1) to[out=-90,in=90] (cmndH2);
% % % % % % % % % % % % % % % %
\draw[thick] (cmndH1) to[out=180+\angmnd,in=180-\angmnd] (mndH1);
\draw[thick] (cmndH1) to[out=-\angmnd,in=\angmnd] (act21);
\draw[thick] (cmndH2) to[out=180+\angmnd,in=\angmnd] (act3);
\draw[thick] (cmndH2) to[out=-\angmnd,in=\angmnd] (act22);
\draw[very thin] (mndA4) to[out=-90,in=180-\angmnd] (act3);
\draw[very thin] (cmndA3) to[out=180+\angmnd,in=180-\angmnd] (mndA4);
\draw[very thin] (cmndA3) to[out=-\angmnd,in=180-\angmnd] (mndA6);
\draw[thick] (act3) to[out=-90,in=\angmnd] (mndH1);
% % % % % % % % % % % % % % % % % % % %
\draw[thick] (mndH1) to[out=-90,in=90-\angio] (outH1);
% % % % % % % % % % % % % % % % % % % % %
\draw[very thin] (elt1) to[out=180,in=180-\angmnd] (act21);
\draw[very thin] (elt12) to[out=180,in=180-\angmnd] (mndA5);
\draw[very thin] (mndA5) to[out=-90,in=180-\angmnd] (act22);
\draw[thick] (act21) to[out=-90,in=180-\angmnd] (mndH2);
\draw[thick] (act22) to[out=-90,in=\angmnd] (mndH2);
\draw[very thin] (elt1) to[out=0,in=90] (cmndA3);
\draw[very thin] (elt12) to[out=0,in=\angmnd] (mndA6);
\draw[very thin] (mndA6) to[out=-90,in=180-\angmnd] (mndA2);
% % % % % % % % % % % % % % % % % % % % %
\draw[thick] (mndH2) to[out=-90,in=90+\angio] (outH2);
\draw[very thin] (mndA2) to[out=-90,in=90-\angio] (outA1);
\end{tikzpicture}
&=
%7
\begin{tikzpicture}[baseline=(current  bounding  box.center)]
\coordinate (inH1) at (-3 * \strx, 10 * \stry);
\coordinate (inA1) at (-1 * \strx, 10 * \stry);
\coordinate (inH2) at (1 * \strx, 10 * \stry);
\coordinate (inA2) at (3 * \strx, 10 * \stry);
\coordinate (cmndA1) at (-0.5 * \strx, 8 * \stry);
\coordinate (cmndA11) at (-1 * \strx, 7 * \stry);
\coordinate (act11) at (-0.75 * \strx, 4 * \stry);
\coordinate (act12) at (1 * \strx, 2 * \stry);
\coordinate (cmndH1) at (-2.5 * \strx, 8 * \stry);
\coordinate (cmndH2) at (1 * \strx, 8 * \stry);
\coordinate (elt2) at (-3 * \strx, 5 * \stry);
\coordinate (act3) at (-2.5 * \strx, -3 * \stry);
\coordinate (act4) at (-2.25 * \strx, 3 * \stry);
\coordinate (mndA4) at (-2.5 * \strx, 1 * \stry);
\coordinate (mndH1) at (-2.5 * \strx, -8 * \stry);
\coordinate (mndH2) at (0.5 * \strx, -8 * \stry);
\coordinate (elt1) at (-3 * \strx, 6 * \stry);
\coordinate (elt12) at (0 * \strx, 2 * \stry);
\coordinate (mndA6) at (-1 * \strx, -2 * \stry);
\coordinate (mndA3) at (-3.5 * \strx, 3 * \stry);
\coordinate (act21) at (-3.5 * \strx, 2 * \stry);
\coordinate (act22) at (1.5 * \strx, -7 * \stry);
\coordinate (mndA2) at (-0.5 * \strx, -7 * \stry);
\coordinate (mndA5) at (0 * \strx, -4.5 * \stry);
\coordinate (outH1) at (-3 * \strx, -10 * \stry);
\coordinate (outA1) at (-1 * \strx, -10 * \stry);
\coordinate (outH2) at (1 * \strx, -10 * \stry);
\coordinate (outA2) at (3 * \strx, -10 * \stry);
\coordinate (mndA1) at (2 * \strx, 1 * \stry);
\coordinate (cmndA2) at (2 * \strx, -3 * \stry);
\draw[very thin] (inA2) to[out=-90-\angio,in=\angmnd] (mndA1);
\draw[very thin] (cmndA1) to[out=-\angmnd,in=180-\angmnd] (mndA1);
\draw[very thin] (mndA1) to[out=-90,in=90] (cmndA2);
\draw[very thin] (cmndA2) to[out=-\angmnd,in=90+\angio] (outA2);
\draw[very thin] (cmndA2) to[out=180+\angmnd,in=\angmnd] (mndA2);
\draw[thick] (inH1) to[out=-90+\angio,in=90] (cmndH1);
\draw[very thin] (inA1) to[out=-90+\angio,in=90] (cmndA1);
\draw[thick] (inH2) to[out=-90-\angio,in=90] (cmndH2);
% % % % % % % % % % % %
\draw[very thin] (cmndA1) to[out=180+\angmnd,in=90] (cmndA11);
\draw[very thin] (cmndA11) to[out=-\angmnd,in=\angmnd] (mndA5);
\draw[very thin] (cmndA11) to[out=180+\angmnd,in=\angmnd] (mndA4);
%\draw[thick] (act1) to[out=-90,in=90] (cmndH2);
% % % % % % % % % % % % % % % %
\draw[thick] (cmndH1) to[out=180+\angmnd,in=180-\angmnd] (mndH1);
\draw[thick] (cmndH1) to[out=-\angmnd,in=\angmnd] (act21);
\draw[thick] (cmndH2) to[out=180+\angmnd,in=\angmnd] (act3);
\draw[thick] (cmndH2) to[out=-\angmnd,in=\angmnd] (act22);
\draw[very thin] (mndA4) to[out=-90,in=180-\angmnd] (act3);
\draw[very thin] (elt2) to[out=0,in=180-\angmnd] (mndA4);
\draw[very thin] (elt2) to[out=180,in=\angmnd] (mndA3);
\draw[thick] (act3) to[out=-90,in=\angmnd] (mndH1);
\draw[very thin] (mndA3) to[out=-90,in=180-\angmnd] (act21);
% % % % % % % % % % % % % % % % % % % %
\draw[thick] (mndH1) to[out=-90,in=90-\angio] (outH1);
% % % % % % % % % % % % % % % % % % % % %
\draw[very thin] (elt1) to[out=180,in=180-\angmnd] (mndA3);
\draw[very thin] (elt12) to[out=180,in=180-\angmnd] (mndA5);
\draw[very thin] (mndA5) to[out=-90,in=180-\angmnd] (act22);
\draw[thick] (act21) to[out=-90,in=180-\angmnd] (mndH2);
\draw[thick] (act22) to[out=-90,in=\angmnd] (mndH2);
\draw[very thin] (elt1) to[out=0,in=180-\angmnd] (mndA6);
\draw[very thin] (elt12) to[out=0,in=\angmnd] (mndA6);
\draw[very thin] (mndA6) to[out=-90,in=180-\angmnd] (mndA2);
% % % % % % % % % % % % % % % % % % % % %
\draw[thick] (mndH2) to[out=-90,in=90+\angio] (outH2);
\draw[very thin] (mndA2) to[out=-90,in=90-\angio] (outA1);
\end{tikzpicture}
\label{d3}\\
&=
%8
\begin{tikzpicture}[baseline=(current  bounding  box.center)]
\coordinate (inH1) at (-3 * \strx, 10 * \stry);
\coordinate (inA1) at (-1 * \strx, 10 * \stry);
\coordinate (inH2) at (1 * \strx, 10 * \stry);
\coordinate (inA2) at (3 * \strx, 10 * \stry);
\coordinate (cmndA1) at (-0.5 * \strx, 8 * \stry);
\coordinate (cmndA11) at (-1 * \strx, 7 * \stry);
\coordinate (act11) at (-0.75 * \strx, 4 * \stry);
\coordinate (act12) at (1 * \strx, 2 * \stry);
\coordinate (cmndH1) at (-2.5 * \strx, 8 * \stry);
\coordinate (cmndH2) at (1 * \strx, 8 * \stry);
\coordinate (elt2) at (-2.5 * \strx, 5 * \stry);
\coordinate (act3) at (-0.5 * \strx, -1 * \stry);
\coordinate (act4) at (-2.25 * \strx, 3 * \stry);
\coordinate (mndA4) at (-2 * \strx, 2 * \stry);
\coordinate (mndH1) at (-2.5 * \strx, -8 * \stry);
\coordinate (mndH2) at (0.5 * \strx, -8 * \stry);
\coordinate (elt1) at (-3.5 * \strx, 2 * \stry);
\coordinate (elt12) at (-2.5 * \strx, 1 * \stry);
\coordinate (mndA6) at (-2 * \strx, -1 * \stry);
\coordinate (mndA3) at (-3.5 * \strx, -2 * \stry);
\coordinate (act21) at (-3 * \strx, -4 * \stry);
\coordinate (act22) at (1.5 * \strx, -7 * \stry);
\coordinate (mndA2) at (-0.5 * \strx, -7 * \stry);
\coordinate (mndA5) at (0 * \strx, -4.5 * \stry);
\coordinate (outH1) at (-3 * \strx, -10 * \stry);
\coordinate (outA1) at (-1 * \strx, -10 * \stry);
\coordinate (outH2) at (1 * \strx, -10 * \stry);
\coordinate (outA2) at (3 * \strx, -10 * \stry);
\coordinate (mndA1) at (2 * \strx, 1 * \stry);
\coordinate (cmndA2) at (2 * \strx, -3 * \stry);
\draw[very thin] (inA2) to[out=-90-\angio,in=\angmnd] (mndA1);
\draw[very thin] (cmndA1) to[out=-\angmnd,in=180-\angmnd] (mndA1);
\draw[very thin] (mndA1) to[out=-90,in=90] (cmndA2);
\draw[very thin] (cmndA2) to[out=-\angmnd,in=90+\angio] (outA2);
\draw[very thin] (cmndA2) to[out=180+\angmnd,in=\angmnd] (mndA2);
\draw[thick] (inH1) to[out=-90+\angio,in=90] (cmndH1);
\draw[very thin] (inA1) to[out=-90+\angio,in=90] (cmndA1);
\draw[thick] (inH2) to[out=-90-\angio,in=90] (cmndH2);
% % % % % % % % % % % %
\draw[very thin] (cmndA1) to[out=180+\angmnd,in=90] (cmndA11);
\draw[very thin] (cmndA11) to[out=-\angmnd,in=\angmnd] (mndA5);
\draw[very thin] (cmndA11) to[out=180+\angmnd,in=\angmnd] (mndA4);
%\draw[thick] (act1) to[out=-90,in=90] (cmndH2);
% % % % % % % % % % % % % % % %
\draw[thick] (cmndH1) to[out=180+\angmnd,in=180-\angmnd] (mndH1);
\draw[thick] (cmndH1) to[out=-\angmnd,in=\angmnd] (act21);
\draw[thick] (cmndH2) to[out=180+\angmnd,in=\angmnd] (act3);
\draw[thick] (cmndH2) to[out=-\angmnd,in=\angmnd] (act22);
\draw[very thin] (mndA4) to[out=-90,in=180-\angmnd] (act3);
\draw[very thin] (elt2) to[out=0,in=180-\angmnd] (mndA4);
\draw[very thin] (elt2) to[out=180,in=\angmnd] (mndA3);
\draw[thick] (act3) to[out=-90,in=\angmnd] (mndH1);
\draw[very thin] (mndA3) to[out=-90,in=180-\angmnd] (act21);
% % % % % % % % % % % % % % % % % % % %
\draw[thick] (mndH1) to[out=-90,in=90-\angio] (outH1);
% % % % % % % % % % % % % % % % % % % % %
\draw[very thin] (elt1) to[out=180,in=180-\angmnd] (mndA3);
\draw[very thin] (elt12) to[out=180,in=180-\angmnd] (mndA5);
\draw[very thin] (mndA5) to[out=-90,in=180-\angmnd] (act22);
\draw[thick] (act21) to[out=-90,in=180-\angmnd] (mndH2);
\draw[thick] (act22) to[out=-90,in=\angmnd] (mndH2);
\draw[very thin] (elt1) to[out=0,in=180-\angmnd] (mndA6);
\draw[very thin] (elt12) to[out=0,in=\angmnd] (mndA6);
\draw[very thin] (mndA6) to[out=-90,in=180-\angmnd] (mndA2);
% % % % % % % % % % % % % % % % % % % % %
\draw[thick] (mndH2) to[out=-90,in=90+\angio] (outH2);
\draw[very thin] (mndA2) to[out=-90,in=90-\angio] (outA1);
\end{tikzpicture}
&=
%9
\begin{tikzpicture}[baseline=(current  bounding  box.center)]
\coordinate (inH1) at (-3 * \strx, 10 * \stry);
\coordinate (inA1) at (-1 * \strx, 10 * \stry);
\coordinate (inH2) at (1 * \strx, 10 * \stry);
\coordinate (inA2) at (3 * \strx, 10 * \stry);
\coordinate (cmndA1) at (-0.5 * \strx, 8 * \stry);
\coordinate (cmndA11) at (-1 * \strx, 7 * \stry);
\coordinate (act11) at (-0.75 * \strx, 4 * \stry);
\coordinate (act12) at (1 * \strx, 2 * \stry);
\coordinate (cmndH1) at (-2.5 * \strx, 8 * \stry);
\coordinate (cmndH2) at (1 * \strx, 8 * \stry);
\coordinate (elt2) at (-2.5 * \strx, 5 * \stry);
\coordinate (act3) at (-0.5 * \strx, -1 * \stry);
\coordinate (act4) at (-2.25 * \strx, 3 * \stry);
\coordinate (mndA4) at (-2 * \strx, 2 * \stry);
\coordinate (mndH1) at (-2.5 * \strx, -8 * \stry);
\coordinate (mndH2) at (0.5 * \strx, -8 * \stry);
\coordinate (elt1) at (-3.5 * \strx, 2 * \stry);
\coordinate (cmndA3) at (-3.5 * \strx, -1 * \stry);
\coordinate (mndA3) at (-3.5 * \strx, -2 * \stry);
\coordinate (act21) at (-3 * \strx, -4 * \stry);
\coordinate (act22) at (1.5 * \strx, -7 * \stry);
\coordinate (mndA2) at (-0.5 * \strx, -7 * \stry);
\coordinate (mndA5) at (0 * \strx, -3.5 * \stry);
\coordinate (outH1) at (-3 * \strx, -10 * \stry);
\coordinate (outA1) at (-1 * \strx, -10 * \stry);
\coordinate (outH2) at (1 * \strx, -10 * \stry);
\coordinate (outA2) at (3 * \strx, -10 * \stry);
\coordinate (mndA1) at (2 * \strx, 1 * \stry);
\coordinate (cmndA2) at (2 * \strx, -3 * \stry);
\draw[very thin] (inA2) to[out=-90-\angio,in=\angmnd] (mndA1);
\draw[very thin] (cmndA1) to[out=-\angmnd,in=180-\angmnd] (mndA1);
\draw[very thin] (mndA1) to[out=-90,in=90] (cmndA2);
\draw[very thin] (cmndA2) to[out=-\angmnd,in=90+\angio] (outA2);
\draw[very thin] (cmndA2) to[out=180+\angmnd,in=\angmnd] (mndA2);
\draw[thick] (inH1) to[out=-90+\angio,in=90] (cmndH1);
\draw[very thin] (inA1) to[out=-90+\angio,in=90] (cmndA1);
\draw[thick] (inH2) to[out=-90-\angio,in=90] (cmndH2);
% % % % % % % % % % % %
\draw[very thin] (cmndA1) to[out=180+\angmnd,in=90] (cmndA11);
\draw[very thin] (cmndA11) to[out=-\angmnd,in=\angmnd] (mndA5);
\draw[very thin] (cmndA11) to[out=180+\angmnd,in=\angmnd] (mndA4);
% % % % % % % % % % % % % % % %
\draw[thick] (cmndH1) to[out=180+\angmnd,in=180-\angmnd] (mndH1);
\draw[thick] (cmndH1) to[out=-\angmnd,in=\angmnd] (act21);
\draw[thick] (cmndH2) to[out=180+\angmnd,in=\angmnd] (act3);
\draw[thick] (cmndH2) to[out=-\angmnd,in=\angmnd] (act22);
\draw[very thin] (mndA4) to[out=-90,in=180-\angmnd] (act3);
\draw[very thin] (elt2) to[out=0,in=180-\angmnd] (mndA4);
\draw[very thin] (elt2) to[out=180,in=\angmnd] (mndA3);
\draw[thick] (act3) to[out=-90,in=\angmnd] (mndH1);
\draw[very thin] (mndA3) to[out=-90,in=180-\angmnd] (act21);
% % % % % % % % % % % % % % % % % % % %
\draw[thick] (mndH1) to[out=-90,in=90-\angio] (outH1);
% % % % % % % % % % % % % % % % % % % % %
\draw[very thin] (elt1) to[out=180,in=90] (cmndA3);
\draw[very thin] (cmndA3) to[out=180+\angmnd,in=180-\angmnd] (mndA3);
\draw[very thin] (cmndA3) to[out=-\angmnd,in=180-\angmnd] (mndA5);
\draw[very thin] (mndA5) to[out=-90,in=180-\angmnd] (act22);
\draw[thick] (act21) to[out=-90,in=180-\angmnd] (mndH2);
\draw[thick] (act22) to[out=-90,in=\angmnd] (mndH2);
\draw[very thin] (elt1) to[out=0,in=180-\angmnd] (mndA2);
% % % % % % % % % % % % % % % % % % % % %
\draw[thick] (mndH2) to[out=-90,in=90+\angio] (outH2);
\draw[very thin] (mndA2) to[out=-90,in=90-\angio] (outA1);
\end{tikzpicture}
\label{d4}\\
&=
%10
\begin{tikzpicture}[baseline=(current  bounding  box.center)]
\coordinate (inH1) at (-3 * \strx, 10 * \stry);
\coordinate (inA1) at (-1 * \strx, 10 * \stry);
\coordinate (inH2) at (1 * \strx, 10 * \stry);
\coordinate (inA2) at (3 * \strx, 10 * \stry);
\coordinate (cmndA1) at (-0.5 * \strx, 8 * \stry);
\coordinate (cmndA11) at (-1 * \strx, 7 * \stry);
\coordinate (act11) at (-0.75 * \strx, 4 * \stry);
\coordinate (act12) at (1 * \strx, 2 * \stry);
\coordinate (cmndH1) at (-2.5 * \strx, 8 * \stry);
\coordinate (cmndH2) at (1 * \strx, 8 * \stry);
\coordinate (cmndA2) at (2 * \strx, 2 * \stry);
\coordinate (elt2) at (-2.5 * \strx, 5 * \stry);
\coordinate (act3) at (-1 * \strx, -0.5 * \stry);
\coordinate (act4) at (-2.25 * \strx, 3 * \stry);
\coordinate (mndH1) at (-2.5 * \strx, -8 * \stry);
\coordinate (mndH2) at (0.5 * \strx, -8 * \stry);
\coordinate (mndA1) at (2 * \strx, -2 * \stry);
\coordinate (elt1) at (-3 * \strx, 1 * \stry);
\coordinate (cmndA3) at (-3.5 * \strx, -1 * \stry);
\coordinate (act21) at (-3 * \strx, -4 * \stry);
\coordinate (act22) at (0.5 * \strx, -4 * \stry);
\coordinate (mndA2) at (-0.5 * \strx, -7 * \stry);
\coordinate (outH1) at (-3 * \strx, -10 * \stry);
\coordinate (outA1) at (-1 * \strx, -10 * \stry);
\coordinate (outH2) at (1 * \strx, -10 * \stry);
\coordinate (outA2) at (3 * \strx, -10 * \stry);
\coordinate (mndA1) at (2 * \strx, 1 * \stry);
\coordinate (cmndA2) at (2 * \strx, -3 * \stry);
\draw[very thin] (inA2) to[out=-90-\angio,in=\angmnd] (mndA1);
\draw[very thin] (cmndA1) to[out=-\angmnd,in=180-\angmnd] (mndA1);
\draw[very thin] (mndA1) to[out=-90,in=90] (cmndA2);
\draw[very thin] (cmndA2) to[out=-\angmnd,in=90+\angio] (outA2);
\draw[very thin] (cmndA2) to[out=180+\angmnd,in=\angmnd] (mndA2);
\draw[thick] (inH1) to[out=-90+\angio,in=90] (cmndH1);
\draw[very thin] (inA1) to[out=-90+\angio,in=90] (cmndA1);
\draw[thick] (inH2) to[out=-90-\angio,in=90] (cmndH2);
% % % % % % % % % % % %
\draw[very thin] (cmndA1) to[out=180+\angmnd,in=90] (cmndA11);
\draw[very thin] (cmndA11) to[out=-\angmnd,in=180-\angmnd] (act12);
\draw[very thin] (cmndA11) to[out=180+\angmnd,in=180-\angmnd] (act11);
%\draw[thick] (act1) to[out=-90,in=90] (cmndH2);
% % % % % % % % % % % % % % % %
\draw[thick] (cmndH1) to[out=180+\angmnd,in=180-\angmnd] (mndH1);
\draw[thick] (cmndH1) to[out=-\angmnd,in=\angmnd] (act4);
\draw[thick] (cmndH2) to[out=180+\angmnd,in=\angmnd] (act11);
\draw[thick] (cmndH2) to[out=-\angmnd,in=\angmnd] (act12);
\draw[thick] (act12) to[out=-90,in=\angmnd] (act22);
\draw[thick] (act11) to[out=-90,in=\angmnd] (act3);
\draw[very thin] (elt2) to[out=0,in=180-\angmnd] (act3);
\draw[very thin] (elt2) to[out=180,in=180-\angmnd] (act4);
\draw[thick] (act3) to[out=-90,in=\angmnd] (mndH1);
\draw[thick] (act4) to[out=-90,in=\angmnd] (act21);
% % % % % % % % % % % % % % % % % % % % % % %
\draw[thick] (mndH1) to[out=-90,in=90-\angio] (outH1);
% % % % % % % % % % % % % % % % % % % % %
\draw[very thin] (elt1) to[out=180,in=90] (cmndA3);
\draw[very thin] (cmndA3) to[out=180+\angmnd,in=180-\angmnd] (act21);
\draw[very thin] (cmndA3) to[out=-\angmnd,in=180-\angmnd] (act22);
\draw[thick] (act21) to[out=-90,in=180-\angmnd] (mndH2);
\draw[thick] (act22) to[out=-90,in=\angmnd] (mndH2);
\draw[very thin] (elt1) to[out=0,in=180-\angmnd] (mndA2);
% % % % % % % % % % % % % % % % % % % % %
\draw[thick] (mndH2) to[out=-90,in=90+\angio] (outH2);
\draw[very thin] (mndA2) to[out=-90,in=90-\angio] (outA1);
\end{tikzpicture}
&=
%11
\begin{tikzpicture} [baseline=(current  bounding  box.center)]
\coordinate (inH1) at (-3 * \strx, 10 * \stry);
\coordinate (inA1) at (-1 * \strx, 10 * \stry);
\coordinate (inH2) at (1 * \strx, 10 * \stry);
\coordinate (inA2) at (3 * \strx, 10 * \stry);
\coordinate (cmndA1) at (-0.5 * \strx, 7 * \stry);
\coordinate (act1) at (-0.5 * \strx, 5 * \stry);
\coordinate (cmndH1) at (-2.5 * \strx, 3 * \stry);
\coordinate (cmndH2) at (-1 * \strx, 3 * \stry);
\coordinate (elt2) at (-2.5 * \strx, 1.5 * \stry);
\coordinate (act3) at (-1.25 * \strx, -0.5 * \stry);
\coordinate (act4) at (-2.25 * \strx, -0.5 * \stry);
\coordinate (mndH1) at (-2.5 * \strx, -3 * \stry);
\coordinate (mndH2) at (-1 * \strx, -3 * \stry);
\coordinate (elt1) at (-1.5 * \strx, -5 * \stry);
\coordinate (act2) at (-1 * \strx, -7 * \stry);
\coordinate (mndA2) at (0.5 * \strx, -7 * \stry);
\coordinate (outH1) at (-3 * \strx, -10 * \stry);
\coordinate (outA1) at (-1 * \strx, -10 * \stry);
\coordinate (outH2) at (1 * \strx, -10 * \stry);
\coordinate (outA2) at (3 * \strx, -10 * \stry);
\coordinate (mndA1) at (1 * \strx, 1 * \stry);
\coordinate (cmndA2) at (1 * \strx, -1 * \stry);
\draw[very thin] (inA2) to[out=-90-\angio,in=\angmnd] (mndA1);
\draw[very thin] (cmndA1) to[out=-\angmnd,in=180-\angmnd] (mndA1);
\draw[very thin] (mndA1) to[out=-90,in=90] (cmndA2);
\draw[very thin] (cmndA2) to[out=-\angmnd,in=90+\angio] (outA2);
\draw[very thin] (cmndA2) to[out=180+\angmnd,in=\angmnd] (mndA2);
\draw[thick] (inH1) to[out=-90+\angio,in=90] (cmndH1);
\draw[very thin] (inA1) to[out=-90+\angio,in=90] (cmndA1);
\draw[thick] (inH2) to[out=-90-\angio,in=\angmnd] (act1);
% % % % % % % % % % % %
\draw[very thin] (cmndA1) to[out=180+\angmnd,in=180-\angmnd] (act1);
% % % % % % % % % % % % % % % %
\draw[thick] (act1) to[out=-90,in=90] (cmndH2);
% % % % % % % % % % % % % % % %
\draw[thick] (cmndH1) to[out=180+\angmnd,in=180-\angmnd] (mndH1);
\draw[thick] (cmndH1) to[out=-\angmnd,in=\angmnd] (act4);
\draw[thick] (cmndH2) to[out=-\angmnd,in=\angmnd] (mndH2);
\draw[thick] (cmndH2) to[out=180+\angmnd,in=\angmnd] (act3);
\draw[very thin] (elt2) to[out=0,in=180-\angmnd] (act3);
\draw[very thin] (elt2) to[out=180,in=180-\angmnd] (act4);
\draw[thick] (act3) to[out=-90,in=\angmnd] (mndH1);
\draw[thick] (act4) to[out=-90,in=180-\angmnd] (mndH2);
% % % % % % % % % % % % % % % % % % % %
\draw[thick] (mndH1) to[out=-90,in=90-\angio] (outH1);
\draw[thick] (mndH2) to[out=-90,in=\angmnd] (act2);
% % % % % % % % % % % % % % % % % % % % %
\draw[very thin] (elt1) to[out=180,in=180-\angmnd] (act2);
\draw[very thin] (elt1) to[out=0,in=180-\angmnd] (mndA2);
% % % % % % % % % % % % % % % % % % % % %
\draw[thick] (act2) to[out=-90,in=90+\angio] (outH2);
\draw[very thin] (mndA2) to[out=-90,in=90-\angio] (outA1);
\end{tikzpicture}
\label{d5}\\
&=
%12
\begin{tikzpicture}[baseline=(current  bounding  box.center)]
\coordinate (inH1) at (-3 * \strx, 10 * \stry);
\coordinate (inA1) at (-1 * \strx, 10 * \stry);
\coordinate (inH2) at (1 * \strx, 10 * \stry);
\coordinate (inA2) at (3 * \strx, 10 * \stry);
\coordinate (mndH) at (-1 * \strx, 1 * \stry);
\coordinate (cmndA1) at (-0.5 * \strx, 6 * \stry);
\coordinate (act1) at (-0.5 * \strx, 4 * \stry);
\coordinate (mndA1) at (1 * \strx, 1 * \stry);
\coordinate (cmndA2) at (1 * \strx, -1 * \stry);
\coordinate (cmndH) at (-1 * \strx, -1 * \stry);
\coordinate (elt) at (-1 * \strx, -3 * \stry);
\coordinate (act2) at (-0.5 * \strx, -6 * \stry);
\coordinate (mndA2) at (0.5 * \strx, -6 * \stry);
\coordinate (outH1) at (-3 * \strx, -10 * \stry);
\coordinate (outA1) at (-1 * \strx, -10 * \stry);
\coordinate (outH2) at (1 * \strx, -10 * \stry);
\coordinate (outA2) at (3 * \strx, -10 * \stry);
\draw[thick] (inH1) to[out=-90+\angio,in=180-\angmnd] (mndH);
\draw[very thin] (inA1) to[out=-90+\angio,in=90] (cmndA1);
\draw[thick] (inH2) to[out=-90-\angio,in=\angmnd] (act1);
\draw[very thin] (inA2) to[out=-90-\angio,in=\angmnd] (mndA1);
% % % % % % % % % % % %
\draw[very thin] (cmndA1) to[out=180+\angmnd,in=180-\angmnd] (act1);
\draw[very thin] (cmndA1) to[out=-\angmnd,in=180-\angmnd] (mndA1);
% % % % % % % % % % % % % % % %
\draw[thick] (act1) to[out=-90,in=\angmnd] (mndH);
% % % % % % % % % % % % % % % %
\draw[thick] (mndH) to[out=-90,in=90] (cmndH);
\draw[very thin] (mndA1) to[out=-90,in=90] (cmndA2);
% % % % % % % % % % % % % % % % % % % %
\draw[thick] (cmndH) to[out=180+\angmnd,in=90-\angio] (outH1);
\draw[thick] (cmndH) to[out=-\angmnd,in=\angmnd] (act2);
\draw[very thin] (cmndA2) to[out=180+\angmnd,in=\angmnd] (mndA2);
\draw[very thin] (cmndA2) to[out=-\angmnd,in=90+\angio] (outA2);
% % % % % % % % % % % % % % % % % % % % %
\draw[very thin] (elt) to[out=180,in=180-\angmnd] (act2);
\draw[very thin] (elt) to[out=0,in=180-\angmnd] (mndA2);
% % % % % % % % % % % % % % % % % % % % %
\draw[thick] (act2) to[out=-90,in=90+\angio] (outH2);
\draw[very thin] (mndA2) to[out=-90,in=90-\angio] (outA1);
\end{tikzpicture}
& \label{d6}
\end{align}
\caption{Diagrams used in the proof of Proposition \ref{prop:semi}. Thick lines stand for $H$, thin lines for $A$, and every branching uniquely determines a relevant arrow in $\mathcal{V}$.}
\end{figure}

The defined (co)multiplication is already part of a (co)monoid structure. The compatibility of counit with unit, counit with multiplication and unit with comultiplication follows directly. What remains to show is the bimonoid axiom, which we have done using the manipulation of string diagrams shown in Figure \ref{fig:diagramsSemi} and described below.

In line (\ref{d1}), after rearrangement we used the compatibility of comultiplication of $A$ with coaction of $A$ on $H$, in the bottom right corner of the middle diagram.

Going from line (\ref{d1}) to line (\ref{d2}) we used the bimonoid axiom for $A$ on the top-left part of the diagram, followed by the (co)associativity for $A$. In the line (\ref{d2}) we used the braiding coelement axiom (\ref{ax:elt1}). When passing from line (\ref{d2}) to line (\ref{d3}) we used the (co)associativity for $A$, together with the bimonoid axiom for $A$, in the right side of the diagram. In the line (\ref{d3}) we used the braiding coelement axiom (\ref{ax:elt2}) on the top-left. Line (\ref{d3}) to (\ref{d4}) involves just a rearrangement, followed by the braiding coelement axiom (\ref{ax:elt3}), on the left of the diagram, in line (\ref{d4}).

Passing from line (\ref{d4}) to line (\ref{d5}) uses the compatibility of comultiplication of $A$ with coaction of $A$ on $H$ at three different places. In line (\ref{d5}) we used that the (co)multiplication of $H$ is an $A$-comodule morphism. Finally, going from line (\ref{d5}) to line (\ref{d6}) uses the bimonoid axiom for $H$.

That (\ref{ap}) is indeed an antipode follows in a similar way. The strategy to show the ``right inverse'' axiom is to use the compatibility of $\alpha_H$ with $\delta_A$, and bimonoid axioms to get all multiplications to the top, and comultiplications to the bottom of the diagram, and then use the right inverse axiom for $A$ multiple times, followed by the right inverse axiom for $H$. The strategy to show the ``left inverse'' axiom is to bring all coactions $\alpha_H$ below $\delta_H$, using the definition of coaction on the product of comodules, followed by the left inverse axiom for $H$, followed by the compatibility of $\eta_A$ with $\delta_A$, and the left inverse axiom for $A$. 
\end{proof}

\begin{prop}
The comparison functor
\begin{align}
\mathrm{Comod}_{\mathcal{W}}(H)&\xrightarrow{F} \mathrm{Comod}_{\mathcal{V}}(H\rtimes A)\\
((B,\alpha_B),\chi_B)&\mapsto (B, 1_H\alpha_B\circ \chi_B) \nonumber \\
(f:C\rightarrow B)&\mapsto (f:C\rightarrow B) \nonumber
\end{align}
is strict monoidal and has a strict monoidal inverse
\begin{align}
\mathrm{Comod}_{\mathcal{V}}(H\rtimes A)&\xrightarrow{F^{-1}}\mathrm{Comod}_{\mathcal{W}}(H) \\
(B,\beta)&\mapsto ((B,\lambda_{AB}\circ\epsilon_H1_A\circ\beta),1_H\lambda_{B}\circ 1_H\epsilon_A\circ\beta) \nonumber \\
(f:C\rightarrow B)&\mapsto (f:C\rightarrow B)\,. \nonumber
\end{align}
\end{prop}
\begin{proof}
Using the dual of Beck's monadicity theorem, we show that the forgetful functor
\begin{align}
\mathrm{Comod}_{\mathcal{W}}(H)&\xrightarrow{\mathcal{U}} \mathcal{V}\\
((B,\alpha_B),\chi_B)&\mapsto B \nonumber \\
(f:C\rightarrow B)&\mapsto (f:C\rightarrow B) \nonumber
\end{align}
is comonadic. Since $\mathcal{U}$ is the composite of the two comonadic functors $\mathcal{U}_H$ and $\mathcal{U}_A$, it has a right adjoint and reflects isomorphisms. The third criterion, not necessarily preserved by composition, is the existence and preservation of $\mathcal{U}$-split equalizers. So, assume the parallel pair 
\begin{align}
f,g:((C,\alpha_C),\chi_C)\rightarrow((B,\alpha_B),\chi_B)
\end{align}
in $\mathrm{Comod}_{\mathcal{W}}(H)$ has a split equalizer 
$h:E\rightarrow C$ in $\mathcal{V}$. That is, there are maps
\begin{align}
B\xrightarrow{t}C \xrightarrow{s} E
\end{align}
satisfying
\begin{align}
s\circ h &= 1_E \label{eq:hs} \\
t\circ f &= 1_C \label{eq:ft}\\
h\circ s &= t\circ g\,. \label{eq:sh}
\end{align}
Comonadicity of $\mathcal{U}_A$ implies that $E$ is an $A$-comodule, with coaction
\begin{align}
\alpha_E=(
E
\xrightarrow{h}
C
\xrightarrow{\alpha_C}
AC
\xrightarrow{1s}
AE
)
\end{align}
and that $h$ is an equalizer in $\mathcal{W}$, but not necessarily split.
The proof involves the following identities (expressing the fact that $h$ is an $A$-comodule morphism):
\begin{align}
1h \circ 1s \circ \alpha_C  \circ  h 
&= 
1t \circ 1g \circ \alpha_C  \circ  h\\
&= 1t \circ 1f \circ \alpha_C  \circ  h\\
&= \alpha_C  \circ  h \label{eq:identA}
\end{align}
where the first equality follows from (\ref{eq:sh}), the second from $f\circ h=g\circ h$ and the fact that $f$ and $g$ are $A$-comodule morphisms, and the third comes from (\ref{eq:ft}). Exactly the same equalities hold with $A$ replaced by $H$, for the same reasons:
\begin{align}
1h \circ 1s \circ \chi_C  \circ  h 
&= 
1t \circ 1g \circ \chi_C  \circ  h\\
&= 1t \circ 1f \circ \chi_C  \circ  h\\
&= \chi_C  \circ  h \,.\label{eq:identH}
\end{align}
Now, the map
\begin{align}
\chi_E := (
E
\xrightarrow{h}
C
\xrightarrow{\chi_C}
HC
\xrightarrow{1s}
HE
)
\end{align}
is an $A$-comodule morphism
\begin{align}
\alpha_{HE}\circ\chi_E 
&\stackrel{\text{def.}}{=} 
\mu_A11\circ 1\sigma 1 \circ \alpha_H\alpha_E \circ 1s \circ \chi_C  \circ  h \nonumber
\\
&\stackrel{\text{(\ref{eq:hs})}}{=} 
\mu_A11\circ 1\sigma 1 \circ 111s \circ 111h \circ \alpha_H\alpha_E \circ 1s \circ \chi_C  \circ  h \nonumber
\\
&\stackrel{h \in \mathcal{W}}{=} 
\mu_A1s\circ 1\sigma 1  \circ \alpha_H\alpha_C \circ 1h \circ 1s \circ \chi_C  \circ  h
\nonumber \\
&\stackrel{\text{(\ref{eq:identH})}}{=}
\mu_A1s\circ 1\sigma 1  \circ \alpha_H\alpha_C \circ \chi_C  \circ  h \nonumber
\\
&\stackrel{}{=}
11s\circ \alpha_{HC} \circ \chi_C  \circ  h
\nonumber \\
&\stackrel{\chi_C\in \mathcal{W}}{=}
11s\circ 1 \chi_C \circ \alpha_C  \circ  h \nonumber \\
&\stackrel{\text{(\ref{eq:identA})}}{=}
11s\circ 1\chi_C\circ 1h\circ 1s\circ \alpha_C \circ h
\nonumber \\
&\stackrel{\text{def.}}{=}
 1\chi_E\circ\alpha_E  \nonumber
\end{align}
compatible with counit and comultiplication on $H$, which follows from the compatibility of $\chi_C$ with counit and comultiplication and equations (\ref{eq:hs}) and (\ref{eq:identH}). Therefore $((E,\alpha_E),\chi_E)$ is an object of $\mathrm{Comod}_{\mathcal{V}}(H)$.

The arrow $h$ is an $H$-comodule morphism, which follows directly from (\ref{eq:identH}). 
To show that it equalizes $f$ and $g$, take $((X,\alpha_X),\chi_X)$ to be an $H$-comodule and $m:X\rightarrow C$ an $H$-comodule morphism satisfying $f\circ m = g\circ m$. 
In $\mathcal{V}$, $s\circ m:X\rightarrow E$ is the unique comparison map, since $h$ is the equalizer of $f$ and $g$. But $s\circ m$ is also an $H$-comodule morphism
\begin{align}
 \chi_E\circ s\circ m 
&\stackrel{\text{def.}}{=} 
1s\circ \chi_C \circ h \circ  s\circ m \nonumber \\
&\stackrel{\text{(\ref{eq:sh})}}{=}
1s\circ \chi_C\circ t \circ   g\circ m \nonumber \\
&\stackrel{}{=}
1s\circ \chi_C\circ t \circ   f\circ m \nonumber \\
&\stackrel{\text{(\ref{eq:ft})}}{=}
1s\circ \chi_C \circ m \nonumber \\
&\stackrel{m\in \mathcal{V}^{H\otimes-}}{=}
1s \circ 1m \circ \chi_X \nonumber
\end{align}
completing the proof that $\mathcal{U}$ is comonadic.

The comparison functor $F$ is strict monoidal: 
the coaction for
$$F(((B,\alpha_B),\chi_B)\otimes ((C,\alpha_C),\chi_C))$$ 
is depicted as follows
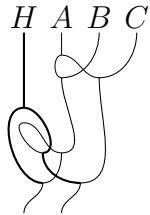
\begin{equation}
\begin{tikzpicture}[baseline=(current  bounding  box.center)]
\def \strx {0.5}
\def \stry {0.2}
\def \angmnd {0}
\def \angio {0}

\node (lblH) at (-1 * \strx, 7 * \stry) {$H$};
\node (lblA) at (-0 * \strx, 7 * \stry) {$A$};
\node (lblB) at (1 * \strx, 7 * \stry) {$B$};
\node (lblC) at (2 * \strx, 7 * \stry) {$C$};

\coordinate (inH) at (-1 * \strx, 6 * \stry);
\coordinate (inA) at (0 * \strx, 6 * \stry);
\coordinate (inB) at (1 * \strx, 6 * \stry);
\coordinate (inC) at (2 * \strx, 6 * \stry);

\coordinate (cmndA) at (0 * \strx, 4.5 * \stry);
\coordinate (actAB1) at (0 * \strx, 3 * \stry);
\coordinate (actAC) at (1 * \strx, 3 * \stry);

\coordinate (actAB2) at (0 * \strx, -2 * \stry);

\coordinate (actHB) at (-0.5 * \strx, -4 * \stry);
\coordinate (actHC) at (0.5 * \strx, -4 * \stry);

\coordinate (cmndH) at (-1 * \strx, 1 * \stry);
\coordinate (elt) at (-1 * \strx, 0 * \stry);
\coordinate (act2) at (-0.5 * \strx, -2 * \stry);
%\coordinate (mndA2) at (0.5 * \strx, -2 * \stry);

\coordinate (outB) at (-1 * \strx, -6 * \stry);
\coordinate (outC) at (0 * \strx, -6 * \stry);

\draw[thick] (inH) to[out=-90+\angio,in=90] (cmndH);
\draw[very thin] (inA) to[out=-90+\angio,in=90] (cmndA);
\draw (inB) to[out=-90-\angio,in=\angmnd] (actAB1)
			to[out=-90,in=\angmnd] (actAB2)
			to[out=-90,in=\angmnd] (actHB)
			to[out=-90,in=90-\angio] (outB);
\draw (inC) to[out=-90-\angio,in=\angmnd] (actAC)
			to[out=-90,in=\angmnd] (actHC)
			to[out=-90,in=90-\angio] (outC);			
% % % % % % % % % % % % % % % % % % % %
\draw[thick] (cmndH) to[out=180+\angmnd,in=180-\angmnd] (actHB);
\draw[thick] (cmndH) to[out=-\angmnd,in=\angmnd] (act2);
\draw[very thin] (cmndA) to[out=180+\angmnd,in=180-\angmnd] (actAB1);
\draw[very thin] (cmndA) to[out=-\angmnd,in=180-\angmnd] (actAC);
% % % % % % % % % % % % % % % % % % % % %
\draw[very thin] (elt) to[out=180,in=180-\angmnd] (act2);
\draw[very thin] (elt) to[out=0,in=180-\angmnd] (actAB2);
% % % % % % % % % % % % % % % % % % % % %
\draw[thick] (act2) to[out=-90,in=180-\angmnd] (actHC);
%\draw[very thin] (mndA2) to[out=-90,in=90-\angio] (outA1);
\end{tikzpicture}
\end{equation}
while the coaction for
$$F((B,\alpha_B),\chi_B)\otimes F((C,\alpha_C),\chi_C)$$ is depicted by
\begin{equation}
\begin{tikzpicture}[baseline=(current  bounding  box.center)]
\def \strx {0.5}
\def \stry {0.2}
\def \angmnd {0}
\def \angio {0}

\node (lblH) at (-1 * \strx, 7 * \stry) {$H$};
\node (lblA) at (-0 * \strx, 7 * \stry) {$A$};
\node (lblB) at (1 * \strx, 7 * \stry) {$B$};
\node (lblC) at (2 * \strx, 7 * \stry) {$C$};

\coordinate (inH) at (-1 * \strx, 6 * \stry);
\coordinate (inA) at (0 * \strx, 6 * \stry);
\coordinate (inB) at (1 * \strx, 6 * \stry);
\coordinate (inC) at (2 * \strx, 6 * \stry);

\coordinate (cmndA) at (0 * \strx, 5 * \stry);
\coordinate (cmndH) at (-1 * \strx, 5 * \stry);
\coordinate (elt) at (-1 * \strx, 4 * \stry);
\coordinate (act2) at (-1 * \strx, 2 * \stry);
\coordinate (mndA2) at (-0.5 * \strx, 2 * \stry);

\coordinate (cross)  at (-0.75 * \strx, 1 * \stry);

\coordinate (actAB) at (-0.5 * \strx, -3 * \stry);
\coordinate (actAC) at (1 * \strx, -3 * \stry);

\coordinate (actHB) at (-1 * \strx, -4 * \stry);
\coordinate (actHC) at (0.5 * \strx, -4 * \stry);

\coordinate (outB) at (-1 * \strx, -6 * \stry);
\coordinate (outC) at (0.5 * \strx, -6 * \stry);

\draw[thick] (inH) to[out=-90+\angio,in=90] (cmndH);
\draw[very thin] (inA) to[out=-90+\angio,in=90] (cmndA);
\draw (inB) to[out=-90-\angio,in=\angmnd] (actAB)
			to[out=-90,in=\angmnd] (actHB)
			to[out=-90,in=90-\angio] (outB);
\draw (inC) to[out=-90-\angio,in=\angmnd] (actAC)
			to[out=-90,in=\angmnd] (actHC)
			to[out=-90,in=90-\angio] (outC);			
% % % % % % % % % % % % % % % % % % % %
\draw[thick] (cmndH) to[out=180+\angmnd,in=180-\angmnd] (actHB);
\draw[thick] (cmndH) to[out=-\angmnd,in=\angmnd] (act2);
\draw[very thin] (cmndA) to[out=180+\angmnd,in=\angmnd] (mndA2);
\draw[very thin] (cmndA) to[out=-\angmnd,in=180-\angmnd] (actAC);
% % % % % % % % % % % % % % % % % % % % %
\draw[very thin] (elt) to[out=180,in=180-\angmnd] (act2);
\draw[very thin] (elt) to[out=0,in=180-\angmnd] (mndA2);
% % % % % % % % % % % % % % % % % % % % %
\draw[thick] (act2) to[out=-90,in=135] (cross)
					to[out=-45,in=180-\angmnd] (actHC);
\draw[very thin] (mndA2) to[out=-90,in=45] (cross)
						 to[out=225,in=180-\angmnd] (actAB);
\end{tikzpicture}
\end{equation}
These are equal since $\alpha_B$ is compatible with comultiplication on $A$.
\end{proof}

\begin{example}
When $\mathcal{V}=(\text{Set}^\text{op},\times)$, the multiplication of $A$ is forced to be the diagonal map, the comultiplication gives $A$ a monoid structure with identity denoted by $e_A$, 
and the only possible cobraiding element is $(e_A,e_A)$. 
An $A$-comodule bimonoid $H$ is the same as a monoid morphism $\phi:A\rightarrow \text{End}(H)$, and $H\rtimes A$ is precisely the semidirect product for monoids generalising the one for groups.
\end{example}

\section{Birings}\label{sec:birings}

In this section we consider two particular types of bimonoids in a braided monoidal additive category. The additivity condition is about existence of direct sums which, as absolute colimits, are preserved by all $\text{Ab}$-enriched functors, in particular tensoring. The naturality of the braiding implies it is compatible with direct sums: 
for $H=A\oplus B$ and $H'=A'\oplus B'$, we have
\begin{equation}
    \sigma_{HH'}=
    \left[\begin{array}{cccc}
    \sigma_{AA'} & 0 & 0 & 0	\\
    0 & 0 & \sigma_{BA'} & 0	\\
    0 & \sigma_{AB'} & 0 & 0	\\
    0 & 0 & 0 & \sigma_{BB'}
    \end{array}\right]
   \end{equation}
which can be concisely written by specifying non-zero components
\begin{equation}
\begin{tikzpicture}[x={(0,-1cm)},y={(-1cm,0)},every node/.style={scale=1},baseline=(current  bounding  box.center)]
\def \strx {0.2};
\def \stry {0.5};
%objects lvl1
\node (DD1) at (-3*\strx,3*\stry) {$AA'$};
\node (DI1) at (-1*\strx,3*\stry) {$AB'$};
\node (ID1) at (1*\strx,3*\stry) {$BA'$};
\node (II1) at (3*\strx,3*\stry) {$BB'$};
%objects lvl3
\node (DD2) at (-3*\strx,-3*\stry) {$A'A$};
\node (DI2) at (-1*\strx,-3*\stry) {$A'B$};
\node (ID2) at (1*\strx,-3*\stry) {$B'A$};
\node (II2) at (3*\strx,-3*\stry) {$B'B$};
%objects lvl2
%arrowsMu
\path[->,font=\tiny,>=angle 90]
(DD1) edge node[above] {$\sigma_{AA'}$} (DD2)
(DI1) edge node[above left] {$\sigma_{AB'}$} (ID2)
(ID1) edge node[below left] {$\sigma_{BA'}$} (DI2)
(II1) edge node[below] {$\sigma_{BB'}$} (II2);
\end{tikzpicture}
\end{equation}
where concatenation is the tensor product and vertical empty space is the direct sum. From here we directly get the following lemma.

\begin{lemma}\label{lemma:0tensor}
If an object $H$ has as a symmetry morphism $$\sigma_{HH}=\pm 1_{HH}\, ,$$ then any decompositions of $H$ into a sum
$$H=\sum_i H_i$$
forces
\begin{align}
\sigma_{H_iH_j}&=0,\text{ for }i\neq j\\
\sigma_{H_iH_i}&=\pm 1_{H_iH_i}\,.
\end{align}
since the components of $\sigma$ are isomorphisms this means that for $i\neq j$
\begin{equation}
H_i\otimes H_j=0\,. \label{eq:CompTen0}
\end{equation}
\end{lemma}

\subsection{The grading Hopf ring}

Let $\mathcal{W}$ be a category which is, in addition, symmetric and has countable coproducts preserved by tensoring. Denote by $\mathbb{Z}\cdot C$ the copower of the object $C\in \mathcal{W}$ by the set of integers $\mathbb{Z}$. In particular,  there is an object
\begin{equation}
Z:=\mathbb{Z}\cdot I\,.
\end{equation}
The addition of integers gives $\mathbb{Z}$ a group (Hopf monoid) structure in $(\mathrm{Set},\times)$, and induces a Hopf ring structure on $Z$, given by
\begin{equation}
\begin{tikzpicture}[scale=2.5, every node/.style={scale=1},baseline=(current  bounding  box.center)]
% % % left hand side
%objects
\node (A1) at (-1,0) {$I\cong \{*\}\cdot I$};
\node (A2) at (0,0) {$\mathbb{Z}\cdot I$};
\node (A3) at (1.5,0) {$(\mathbb{Z}\times \mathbb{Z})\cdot I\cong Z\otimes Z\,.$};
%arrows
\path[transform canvas={yshift=1mm},<-,font=\scriptsize,>=angle 90]
(A1) edge node[above] {$!\cdot I$} (A2);
\path[transform canvas={yshift=-1mm},->,font=\scriptsize,>=angle 90]
(A1)edge node[below] {$0\cdot I$} (A2);
\path[transform canvas={yshift=1mm},->,font=\scriptsize,>=angle 90]
(A2) edge node[above] {$\Delta\cdot I$} (A3);
\path[transform canvas={yshift=-1mm},<-,font=\scriptsize,>=angle 90]
(A2)edge node[below] {$+\cdot I$} (A3);
\end{tikzpicture}
\end{equation}
Tensoring with $Z$ gives a functor isomorphic to taking a copower by $\mathbb{Z}$
\begin{align*}
Z\otimes C& = (\mathbb{Z}\cdot I)\otimes C\\
	&\cong \mathbb{Z}\cdot (I\otimes C)\\
	&\cong  \mathbb{Z}\cdot C\,.
\end{align*}

Since $\mathcal{W}$ is symmetric monoidal, the category 
$\mathrm{Comod}_{\mathcal{W}}(Z)$ of $Z$-comodules is monoidal,
and becomes braided on using the braiding coelement $\gamma$ depicted by
\begin{equation}\label{diag:coel}
\begin{tikzpicture}[xscale=0.6, every node/.style={scale=1},baseline=(current  bounding  box.center)]
% % % left hand side
%objects
\node (A1) at (-3.5,0.7) {$(\mathbb{Z}\times \mathbb{Z})\cdot I$};
\node (A2) at (0,0.7) {$I$};
\node (A3) at (-3.5,-0.7) {$I$};
%arrows
\path[<-,font=\scriptsize,>=angle 90]
(A1) edge node[transform canvas={yshift=-1mm},above] {$c_{ij}$} (A2);
\path[transform canvas={yshift=0mm},->,font=\scriptsize,>=angle 90]
(A2) edge node[below right] {$(-1)^{ij}$} (A3);
\path[transform canvas={yshift=0mm},->,font=\scriptsize,>=angle 90]
(A1)edge node[left] {$\gamma$} (A3);
\end{tikzpicture}
\end{equation}
Arrows $c_{ij}$ denote coproduct coprojections, and $\gamma$ satisfies the coelement axioms (\ref{ax:elt1})-(\ref{ax:elt3}):
\begin{equation}
\begin{tikzpicture}[yscale=1,baseline=(current  bounding  box.center)]
\def \strx {0.5}
\def \stry {0.2}
\def \angmnd {0}
\node (lhs) at (-1,0) {\begin{tikzpicture}[every node/.style={scale=0.6}]
\coordinate (inA) at (-1 * \strx, 5 * \stry);
\coordinate (elt) at (1 * \strx, 3 * \stry);
\coordinate (cmndA) at (-1 * \strx, 3 * \stry);
\coordinate (mndA1) at (-0.5 * \strx, -1 * \stry);
\coordinate (mndA2) at (-1 * \strx, -3 * \stry);
\coordinate (outA1) at (-1 * \strx, -5 * \stry);
\coordinate (outA2) at (1 * \strx, -5 * \stry);
\node at (outA1) [left] {$i$};
\node at (outA2) [left] {$j$};
\node at (inA) [right] {$(i+j)(-1)^{ij}$};
\draw[very thin] (inA) to[out=-90,in=90] (cmndA);
\draw[very thin] (elt) to[out=180,in=\angmnd] (mndA1);
\draw[very thin] (elt) to[out=0,in=\angmnd] (mndA2);
\draw[very thin] (cmndA) to[out=180+\angmnd,in=180-\angmnd] (mndA2);
\draw[very thin] (cmndA) to[out=-\angmnd,in=180-\angmnd] (mndA1);
\draw[very thin] (mndA1) to[out=-90,in=90] (outA2);
\draw[very thin] (mndA2) to[out=-90,in=90] (outA1);
\end{tikzpicture}};
\node (eq) at (0,0) {=};
\node (rhs) at (1,0) {\begin{tikzpicture}[every node/.style={scale=0.6}]
\coordinate (inA) at (1 * \strx, 5 * \stry);
\coordinate (elt) at (-1 * \strx, 3 * \stry);
\coordinate (cmndA) at (1 * \strx, 3 * \stry);
\coordinate (mndA1) at (-1 * \strx, -1 * \stry);
\coordinate (mndA2) at (1 * \strx, -1 * \stry);
\coordinate (outA1) at (-1 * \strx, -5 * \stry);
\coordinate (outA2) at (1 * \strx, -5 * \stry);
\node at (outA1) [right] {$i$};
\node at (outA2) [left] {$j$};
\node at (inA) [left] {$(-1)^{ji}(j+i)$};
\draw[very thin] (inA) to[out=-90,in=90] (cmndA);
\draw[very thin] (elt) to[out=180,in=180-\angmnd] (mndA1);
\draw[very thin] (elt) to[out=0,in=180-\angmnd] (mndA2);
\draw[very thin] (cmndA) to[out=180+\angmnd,in=\angmnd] (mndA1);
\draw[very thin] (cmndA) to[out=-\angmnd,in=\angmnd] (mndA2);
\draw[very thin] (mndA1) to[out=-90,in=90] (outA2);
\draw[very thin] (mndA2) to[out=-90,in=90] (outA1);
\end{tikzpicture}};
\end{tikzpicture}
\end{equation}
\begin{equation}
\begin{tikzpicture}[baseline=(current  bounding  box.center)]
\def \strx {0.5}
\def \stry {0.8}
\def \angmnd {0}
\node (lhs) at (-1,0) {\begin{tikzpicture}[every node/.style={scale=0.6}]
\coordinate (elt) at (0 * \strx, 1 * \stry);
\coordinate (cmndA) at (1 * \strx, 0.5 * \stry);
\coordinate (outA1) at (-1 * \strx, -1 * \stry);
\coordinate (outA2) at (0.5 * \strx, -1 * \stry);
\coordinate (outA3) at (1 * \strx, -1 * \stry);
\node at (outA1) [left] {$i$};
\node at (outA2) [left] {$j$};
\node at (outA3) [right] {$k$};
\node at (elt) [above] {$(-1)^{i(j+k)}$};
\draw[very thin] (elt) to[out=0,in=90] (cmndA);
\draw[very thin] (elt) to[out=180,in=90] (outA1);
\draw[very thin] (cmndA) to[out=180,in=90] (outA2);
\draw[very thin] (cmndA) to[out=0,in=90] (outA3);
\end{tikzpicture}};
\node (eq) at (0,0) {=};
\node (rhs) at (1,0) {\begin{tikzpicture}[every node/.style={scale=0.6}]
\coordinate (elt1) at (0 * \strx, 1 * \stry);
\coordinate (elt2) at (0 * \strx, 0.5 * \stry);
\coordinate (mndA) at (-1 * \strx, -0.5 * \stry);
\coordinate (outA1) at (-1 * \strx, -1 * \stry);
\coordinate (outA2) at (0.5 * \strx, -1 * \stry);
\coordinate (outA3) at (1 * \strx, -1 * \stry);
\node at (outA1) [left] {$i$};
\node at (outA2) [left] {$j$};
\node at (outA3) [right] {$k$};
\node at (elt) [above] {$(-1)^{ij}(-1)^{ik}$};
\draw[very thin] (elt1) to[out=180,in=180-\angmnd] (mndA);
\draw[very thin] (elt1) to[out=0,in=90] (outA3);
\draw[very thin] (elt2) to[out=180,in=\angmnd] (mndA);
\draw[very thin] (elt2) to[out=0,in=90] (outA2);
\draw[very thin] (mndA) to[out=-90,in=90] (outA1);
\end{tikzpicture}};
\end{tikzpicture}
\end{equation}
\begin{equation}
\begin{tikzpicture}[baseline=(current  bounding  box.center)]
\def \strx {0.5}
\def \stry {0.8}
\def \angmnd {0}
\node (lhs) at (-1,0) {\begin{tikzpicture}[every node/.style={scale=0.6}]
\coordinate (elt1) at (-1 * \strx, 1 * \stry);
\coordinate (elt2) at (1 * \strx, 1 * \stry);
\coordinate (mndA) at (0.5 * \strx, -0.5 * \stry);
\coordinate (outA1) at (-1.5 * \strx, -1 * \stry);
\coordinate (outA2) at (-0.5 * \strx, -1 * \stry);
\coordinate (outA3) at (0.5 * \strx, -1 * \stry);
\node at (outA1) [left] {$i$};
\node at (outA2) [left] {$j$};
\node at (outA3) [right] {$k$};
\node at (elt) [above] {$(-1)^{ik}(-1)^{jk}$};
\draw[very thin] (elt1) to[out=0,in=180-\angmnd] (mndA);
\draw[very thin] (elt1) to[out=180,in=90] (outA1);
\draw[very thin] (elt2) to[out=0,in=\angmnd] (mndA);
\draw[very thin] (elt2) to[out=180,in=90] (outA2);
\draw[very thin] (mndA) to[out=-90,in=90] (outA3);
\end{tikzpicture}};
\node (eq) at (0,0) {=};
\node (rhs) at (1,0) {\begin{tikzpicture}[every node/.style={scale=0.6}]
\coordinate (elt) at (0 * \strx, 1 * \stry);
\coordinate (cmndA) at (-1 * \strx, 0.5 * \stry);
\coordinate (outA1) at (-1 * \strx, -1 * \stry);
\coordinate (outA2) at (-0.5 * \strx, -1 * \stry);
\coordinate (outA3) at (1 * \strx, -1 * \stry);
\node at (outA1) [left] {$i$};
\node at (outA2) [right] {$j$};
\node at (outA3) [right] {$k$};
\node at (elt) [above] {$(-1)^{(i+j)k}$};
\draw[very thin] (elt) to[out=180,in=90] (cmndA);
\draw[very thin] (elt) to[out=0,in=90] (outA3);
\draw[very thin] (cmndA) to[out=180,in=90] (outA1);
\draw[very thin] (cmndA) to[out=0,in=90] (outA2);
\end{tikzpicture}};
\end{tikzpicture} 
\end{equation}

$\mathrm{Comod}_{\mathcal{W}}(Z)$ inherits direct sums: if $B\xrightarrow{b}\mathbb{Z}\cdot B$ and $C\xrightarrow{c}\mathbb{Z}\cdot C$ are $Z$-comodules, then
\begin{equation*}
\begin{tikzpicture}[x={(0,-1cm)},y={(-1cm,0)},every node/.style={scale=1},baseline=(current  bounding  box.center)]
\def \strx {0.25};
\def \stry {0.3};
%objects lvl1
\node (DD1) at (-1*\strx,3*\stry) {$B$};
\node (II1) at (1*\strx,3*\stry) {$C$};
%objects lvl3
\node (DD2) at (-1*\strx,-3*\stry) {$\mathbb{Z}\cdot B$};
\node (II2) at (1*\strx,-3*\stry) {$\mathbb{Z}\cdot C$};
%objects lvl2
%arrowsMu
\path[->,font=\tiny,>=angle 90]
(DD1) edge node[above] {$b$} (DD2)
(II1) edge node[above] {$c$} (II2);
\end{tikzpicture}
\end{equation*}
is a $Z$-comodule as well, and the braiding induced from the braiding coelement $\gamma$ is compatible with direct sums. 

\begin{example}\label{ex:GAb}
When $\mathcal{W}=\mathrm{Ab}$, the biring $Z=\mathbb{Z}[x,x^{-1}]$ is the Laurent polynomial ring with integer coefficients. The coring structure is given by $1\mapsfrom x \mapsto x\otimes x$. Then
\begin{align*}
\mathrm{GAb} \rightarrow & \mathrm{Comod}_{\mathcal{W}}(Z) \\
C \mapsto & \Sigma C_n \xrightarrow{\xi} Z\otimes  \Sigma C_n\\
& c\in C_n\mapsto x^n\otimes c 
\end{align*}
is an equivalence of categories. Consider a $Z$-comodule
\begin{align*}
B&\xrightarrow{\beta} Z\otimes B\\
b&\mapsto \Sigma_i x^i\otimes \beta^{(b)}_i 
\end{align*}
$\beta$ being a group homomorphism ensures that
\begin{equation*}
\beta^{(b)}_i +\beta^{(b')}_i=\beta^{(b+b')}_i\text{ and } 
\beta^{(0)}_i = 0
\end{equation*}
which enable us to define abelian subgroups
\begin{equation*}
B_i=\{\beta^{(b)}_i|b\in B\}
\end{equation*}
while the compatibility with counit and comultiplication gives
\begin{align*}
\Sigma_i \beta^{(b)}_i&=b \\
\beta^{(\beta^{(b)}_i)}_j&=\delta_{i,j}\beta^{(b)}_i
\end{align*}
which ensure that 
\begin{equation*}
B=\Sigma_i B_i \ .
\end{equation*}
Here $\delta_{i,j}$ denotes the Kronecker delta.

The braiding coelement (\ref{diag:coel}) corresponds to the group homomorphism
\begin{align*}
\mathbb{Z}[x,x^{-1}]\otimes \mathbb{Z}[x,x^{-1}] &\xrightarrow{\gamma} \mathbb{Z}\\
x^i\otimes x^j &\mapsto (-1)^{ij}
\end{align*}
and gives a braiding (symmetry, in fact) in GAb.
\end{example}

\subsection{The differential Hopf ring}

Let $\mathcal{W}$ be a braided monoidal additive category.

\begin{prop}\label{prop:HopffromD}
An object $D$ with braiding $\sigma_{DD}=-1_{DD}$ induces a 
Hopf ring $H=D\oplus I$, whose monoid structure
$HH\xrightarrow{\mu} H\xleftarrow{\eta}I$
has non-zero components
\begin{equation} \label{eq:monStr}
\begin{tikzpicture}[x={(0,-1cm)},y={(-1cm,0)},baseline=(current  bounding  box.center)]
\def \strx {0.2};
\def \stry {1};
%objects lvl1
\node (DD1) at (-3*\strx,3*\stry) {$DD$};
\node (DI1) at (-1*\strx,3*\stry) {$DI$};
\node (ID1) at (1*\strx,3*\stry) {$ID$};
\node (II1) at (3*\strx,3*\stry) {$II$};
%
%objects lvl2
\node (D) at (-1*\strx, 0*\stry) {$D$};
\node (I) at (1*\strx, 0*\stry) {$I$};
%
%objects lvl3
\node (I1) at (0*\strx, -3*\stry) {$I\,.$};
%\node (dot) at (0*\strx, -3.2*\stry) {$.$};
%
%arrowsMu
\path[->,font=\tiny,>=angle 90]
(DI1) edge node[transform canvas={yshift=-1mm},above] {$\rho:=\rho_D$} (D)
(ID1) edge node[transform canvas={yshift=0.7mm},below] {$\lambda:=\lambda_D$} (D)
(II1) edge node[below] {$i:=\rho_I=\lambda_I$} (I);
%arrowsEta
\path[->,font=\tiny,>=angle 90]
(I1) edge node[above] {$1$} (I);
\end{tikzpicture}
\end{equation}
the comonoid structure $(\Delta,\epsilon)$ has inverses of (\ref{eq:monStr}) as non-zero components, and the antipode is \[
   s =
  \left[ {\begin{array}{cc}
   -1 & 0 \\
   0 & 1 \\
  \end{array} } \right]\,.
\]
\end{prop}
\begin{proof}
The (co)associativity and (co)unit axioms follow from coherence for monoidal categories, after noting that a component is non-zero if and only if it contains either one $D$ in its source and target, or none.

The compatibility of unit with counit and comultiplication is obvious.

The bimonoid axiom 
\begin{eqnarray}\label{eq:bimon1}
\begin{aligned}
HH\xrightarrow{\Delta\Delta}HHHH\xrightarrow{1\sigma 1}HHHH\xrightarrow{\mu\mu}HH\\
=HH\xrightarrow{\mu}H\xrightarrow{\Delta}HH
\end{aligned}
\end{eqnarray}
imposes that
\begin{equation}
\begin{tikzpicture}[x={(0,-1cm)},y={(-1cm,0)},baseline=(current  bounding  box.center)]
\def \strx {0.35};
\def \stry {1.5};
%objects lvl1
\node (DD1) at (-3*\strx,3*\stry) {$DD$};
\node (DI1) at (-1*\strx,3*\stry) {$DI$};
\node (ID1) at (1*\strx,3*\stry) {$ID$};
\node (II1) at (3*\strx,3*\stry) {$II$};
%
%objects lvl2
%\node (DDDD1) at (-15*\strx,1*\stry) {$DDDD$};
%\node (DDDI1) at (-13*\strx,1*\stry) {$DDDI$};
%\node (DDID1) at (-11*\strx,1*\stry) {$DDID$};
%\node (DDII1) at (-9*\strx,1*\stry) {$DDII$};
%\node (DIDD1) at (-7*\strx,1*\stry) {$DIDD$};
%\node (DIDI1) at (-9*\strx,1*\stry) {$DIDI$};
\node (DIID1) at (-6*\strx,1*\stry) {$DIID$};
\node (DIII1) at (-2*\strx,1*\stry) {$DIII$};
%\node (IDDD1) at (1*\strx,1*\stry) {$IDDD$};
\node (IDDI1) at (-4*\strx,1*\stry) {$IDDI$};
%\node (IDID1) at (-11*\strx,1*\stry) {$IDID$};
\node (IDII1) at (0*\strx,1*\stry) {$IDII$};
%\node (IIDD1) at (9*\strx,1*\stry) {$IIDD$};
\node (IIDI1) at (2*\strx,1*\stry) {$IIDI$};
\node (IIID1) at (4*\strx,1*\stry) {$IIID$};
\node (IIII1) at (6*\strx,1*\stry) {$IIII$};
%
%objects lvl3
%\node (DDDD2) at (-15*\strx,-1*\stry) {$DDDD$};
%\node (DDDI2) at (-13*\strx,-1*\stry) {$DDDI$};
%\node (DDID2) at (-11*\strx,-1*\stry) {$DDID$};
%\node (DDII2) at (-9*\strx,-1*\stry) {$DDII$};
%\node (DIDD2) at (-7*\strx,-1*\stry) {$DIDD$};
%\node (DIDI2) at (-5*\strx,-1*\stry) {$DIDI$};
\node (DIID2) at (-6*\strx,-1*\stry) {$DIID$};
\node (DIII2) at (-2*\strx,-1*\stry) {$DIII$};
%\node (IDDD2) at (1*\strx,-1*\stry) {$IDDD$};
\node (IDDI2) at (-4*\strx,-1*\stry) {$IDDI$};
%\node (IDID2) at (-7*\strx,-1*\stry) {$IDID$};
\node (IDII2) at (0*\strx,-1*\stry) {$IDII$};
%\node (IIDD2) at (-11*\strx,-1*\stry) {$IIDD$};
\node (IIDI2) at (2*\strx,-1*\stry) {$IIDI$};
\node (IIID2) at (4*\strx,-1*\stry) {$IIID$};
\node (IIII2) at (6*\strx,-1*\stry) {$IIII$};
%objects lvl4
\node (DD2) at (-3*\strx,-3*\stry) {$DD$};
\node (DI2) at (-1*\strx,-3*\stry) {$DI$};
\node (ID2) at (1*\strx,-3*\stry) {$ID$};
\node (II2) at (3*\strx,-3*\stry) {$II$};
%
%objects lvl2
%arrowsDD
\path[->,font=\tiny,>=angle 90]
%(DD1) edge node[above right] {$\rho^{-1}\rho^{-1}$} (DIDI1)
(DD1) edge node[fill=white,rounded corners=2pt,inner sep=1pt] {$\rho^{-1}\lambda^{-1}$} (DIID1)
(DD1) edge node[fill=white,rounded corners=2pt,inner sep=1pt,below] {$\lambda^{-1}\rho^{-1}$} (IDDI1)
%(DD1) edge node[above right] {$\lambda^{-1}\lambda^{-1}$} (IDID1)
(DI1) edge node[fill=white,rounded corners=2pt,inner sep=1pt] {$\rho^{-1}i^{-1}$} (DIII1)
(DI1) edge node[fill=white,rounded corners=2pt,inner sep=1pt] {$\lambda^{-1}i^{-1}$} (IDII1)
(ID1) edge node[fill=white,rounded corners=2pt,inner sep=1pt] {$i^{-1}\rho^{-1}$} (IIDI1)
(ID1) edge node[fill=white,rounded corners=2pt,inner sep=1pt,below] {$i^{-1}\lambda^{-1}$} (IIID1)
(II1) edge node[fill=white,rounded corners=2pt,inner sep=1pt,below] {$i^{-1}i^{-1}$} (IIII1);
%arrowsSigma
\path[->,font=\tiny,>=angle 90]
%(DIDI1) edge node[above] {$1\sigma_{ID}1$} (DDII2)
(DIID1) edge node[above] {$1\sigma_{II}1$} (DIID2)
(IDDI1) edge node[above] {$1\sigma_{DD}1$} (IDDI2)
%(IDID1) edge node[above] {$1\sigma_{DI}1$} (IIDD2)
(DIII1) edge node[above] {$1\sigma_{II}1$} (DIII2)
(IDII1) edge node[transform canvas={xshift=-1mm},above] {$1\sigma_{DI}1$} (IIDI2)
(IIDI1) edge node[transform canvas={xshift=-1mm},below] {$1\sigma_{ID}1$} (IDII2)
(IIID1) edge node[above] {$1\sigma_{II}1$} (IIID2)
(IIII1) edge node[above] {$1\sigma_{II}1$} (IIII2);
%arrowsMuMu
\path[<-,font=\tiny,>=angle 90]
%(DD2) edge node[above left] {$\rho\rho$} (DIDI2)
(DD2) edge node[fill=white,rounded corners=2pt,inner sep=1pt, above left] {$\rho\lambda$} (DIID2)
(DD2) edge node[fill=white,rounded corners=2pt,inner sep=1pt, left] {$\lambda\rho$} (IDDI2)
%(DD2) edge node[above left] {$\lambda\lambda$} (IDID2)
(DI2) edge node[fill=white,rounded corners=2pt,inner sep=1pt, left] {$\rho i$} (DIII2)
(DI2) edge node[fill=white,rounded corners=2pt,inner sep=1pt, left] {$\lambda i$} (IDII2)
(ID2) edge node[fill=white,rounded corners=2pt,inner sep=1pt, left] {$i\rho$} (IIDI2)
(ID2) edge node[fill=white,rounded corners=2pt,inner sep=1pt, below left] {$i\lambda$} (IIID2)
(II2) edge node[fill=white,rounded corners=2pt,inner sep=1pt, below left] {$i\lambda$} (IIII2);
\end{tikzpicture}
\end{equation}
equals
\begin{equation}\label{eq:bimon4}
\begin{tikzpicture}[x={(0,-1cm)},y={(-1cm,0)},baseline=(current  bounding  box.center)]
\def \strx {0.2};
\def \stry {1};
%objects lvl1
\node (DD1) at (-3*\strx,3*\stry) {$DD$};
\node (DI1) at (-1*\strx,3*\stry) {$DI$};
\node (ID1) at (1*\strx,3*\stry) {$ID$};
\node (II1) at (3*\strx,3*\stry) {$II$};
%
%objects lvl2
\node (D) at (-1*\strx, 0*\stry) {$D$};
\node (I) at (1*\strx, 0*\stry) {$I$};
%
%objects lvl3
\node (DD2) at (-3*\strx,-3*\stry) {$DD$};
\node (DI2) at (-1*\strx,-3*\stry) {$DI$};
\node (ID2) at (1*\strx,-3*\stry) {$ID$};
\node (II2) at (3*\strx,-3*\stry) {$II$};
%objects lvl2
%arrowsMu
\path[->,font=\tiny,>=angle 90]
(DI1) edge node[fill=white,rounded corners=2pt,inner sep=1pt, above left] {$\rho$} (D)
(ID1) edge node[fill=white,rounded corners=2pt,inner sep=1pt, left] {$\lambda$} (D)
(II1) edge node[fill=white,rounded corners=2pt,inner sep=1pt, below left] {$i$} (I);
%arrowsD
\path[<-,font=\tiny,>=angle 90]
(DI2) edge node[fill=white,rounded corners=2pt,inner sep=1pt, above right] {$\rho^{-1}$} (D)
(ID2) edge node[fill=white,rounded corners=2pt,inner sep=1pt, right] {$\lambda^{-1}$} (D)
(II2) edge node[fill=white,rounded corners=2pt,inner sep=1pt, below right] {$i^{-1}$} (I);
\end{tikzpicture}
\end{equation}
which follows from $\sigma_{DD}=-1$, braiding coherences \cite{Joyal1993}:  $\sigma_{DI}=\lambda^{-1}\circ\rho$, $\sigma_{ID}=\rho^{-1}\circ\lambda$ and $\sigma_{II}=1$, and coherences for unit and associator.
%\begin{equation}
%(\rho\lambda)\circ(1\sigma_{II}1)\circ(\rho^{-1}\lambda^{-1})+
%(\lambda\rho)\circ(1\sigma_{DD}1)\circ(\lambda^{-1}\rho^{-1})=0_{DD}
%\end{equation}

Finally, the Hopf axioms hold, for example the left inverse part gives
\begin{equation}
\begin{tikzpicture}[x={(0,-1cm)},y={(-1cm,0)},baseline=(current  bounding  box.center)]
\def \strx {0.2};
\def \stry {1};

%objects lvl1
\node (D1) at (-1*\strx, 3*\stry) {$D$};
\node (I1) at (1*\strx, 3*\stry) {$I$};
%%objects lvl1%
%\node (DD1) at (-3*\strx,1*\stry) {$DD$};
\node (DI1) at (-2*\strx,1*\stry) {$DI$};
\node (ID1) at (0*\strx,1*\stry) {$ID$};
\node (II1) at (2*\strx,1*\stry) {$II$};

%objects lvl3
%\node (DD2) at (-3*\strx,-1*\stry) {$DD$};
\node (DI2) at (-2*\strx,-1*\stry) {$DI$};
\node (ID2) at (0*\strx,-1*\stry) {$ID$};
\node (II2) at (2*\strx,-1*\stry) {$II$};
%objects lvl4
\node (D2) at (-1*\strx, -3*\stry) {$D$};
\node (I2) at (1*\strx, -3*\stry) {$I$};
%arrowsMu
\path[<-,font=\tiny,>=angle 90]
(DI1) edge node[fill=white,rounded corners=2pt,inner sep=1pt, above] {$\rho^{-1}$} (D1)
(ID1) edge node[fill=white,rounded corners=2pt,inner sep=1pt] {$\lambda^{-1}$} (D1)
(II1) edge node[fill=white,rounded corners=2pt,inner sep=1pt, below] {$i^{-1}$} (I1);
%arrows
\path[->,font=\tiny,>=angle 90]
(DI1) edge node[above] {$-1$} (DI2)
(ID1) edge node[above] {$1$} (ID2)
(II1) edge node[above] {$1$} (II2);
%arrowsD
\path[->,font=\tiny,>=angle 90]
(DI2) edge node[fill=white,rounded corners=2pt,inner sep=1pt, above] {$\rho$} (D2)
(ID2) edge node[fill=white,rounded corners=2pt,inner sep=1pt] {$\lambda$} (D2)
(II2) edge node[fill=white,rounded corners=2pt,inner sep=1pt, below] {$i$} (I2);
\end{tikzpicture}
\end{equation}
equals
\begin{equation}
\begin{tikzpicture}[x={(0,-1cm)},y={(-1cm,0)},baseline=(current  bounding  box.center)]
\def \strx {0.2};
\def \stry {1};

%objects lvl1
\node (D1) at (-1*\strx, 3*\stry) {$D$};
\node (I1) at (1*\strx, 3*\stry) {$I$};
%%objects lvl1%
%\node (DD1) at (-3*\strx,1*\stry) {$DD$};
\node (I) at (0*\strx,0*\stry) {$I$};
%objects lvl4
\node (D2) at (-1*\strx, -3*\stry) {$D$};
\node (I2) at (1*\strx, -3*\stry) {$I$};
%arrowsMu
\path[->,font=\tiny,>=angle 90]
(I1) edge node[above] {$1$} (I)
(I) edge node[above] {$1$} (I2);
\end{tikzpicture}
\end{equation}
\end{proof}

%\begin{example}
%When $\mathcal{V}=\text{GAb}$ and $D=S\mathbb{Z}$, the conditions for Proposition \ref{prop:HopffromD} are satisfied and $H=T\mathbb{Z}.$
%\end{example}

%\subsubsection{Hopf coring algebras}
%
%Here we assume $\mathcal{V}$ is coclosed.
%Tensoring with a Hopf monoid gives a monad $G$ on $\mathcal{V}$ and induces monoidal closed structure on the category of EM-algebras $\mathcal{V}^T$.
%Explicitly, objects of $\mathcal{V}^T$ are arrows of the form
%\begin{equation}
%DA\oplus A\xrightarrow{[d\,\,1]}A,\text{ such that }
%d\circ 1_Dd=0_{DDA,A}
%\end{equation}
%An algebra morphism is an arrow $A\xrightarrow{f}B$ satisfying
%\begin{equation}\label{eq:morph}
%f\circ d_A=d_B \circ 1_Df
%\end{equation}
%Tensor product of $(A,d_A)$ and $(B,d_B)$ is $(AB,d_A1+(1d_B)\circ(\sigma_{DA}1))$. The unit for the tensor product is the algebra $(I,0).$
%Internal hom $\mathcal{V}^T((A,d_A),(B,d_B))=([A,B],d_{[A,B]})$, where $d_{[A,B]}$ is adjunct of the following map
%\begin{equation}
%D[A,B]A\xrightarrow{d_B\circ 1\text{Ev}-\text{Ev}\circ1d_A\circ \sigma_{D,[A,B]}1}B
%\end{equation}
%
%There is an obvious algebra structure on $D$, $(D,0)$, and tensoring with it gives the desuspension functor on $\mathcal{V}^T$
%\begin{equation}
%S^*(A,d)=(DA,-1_Dd)
%\end{equation}
%which, due to the braiding coherence and $\sigma_{DD}=-1$, satisfies
%\begin{equation}
%S^*(A\otimes B)=S^*A\otimes B
%\end{equation}

\subsubsection{When $\mathcal{W}$ is a category of comodules}

Let $\mathcal{V}$ be a symmetric monoidal additive category, and $A$ a biring there, with a braiding coelement $A\otimes A\xrightarrow{\gamma}I$. Take $\mathcal{W}=\mathrm{Comod}_{\mathcal{V}}(A)$. Now $D$ as an $A$-comodule is an object of $\mathcal{V}$, together with a coaction $d: D\rightarrow A\otimes D$ satisfying
\begin{equation}\label{eq:Dcondition}
(\gamma 11)\circ(1\sigma_{AD}1)\circ(dd)\circ\sigma_{DD}=-1_{DD}
\end{equation}
where on the left we have the braiding in $\mathcal{W}$ and $\sigma$ is the symmetry in $\mathcal{V}$. 
From Proposition \ref{prop:HopffromD}, we have that $H=D\oplus I$ with the coaction
\begin{equation}
\begin{tikzpicture}[x={(0,-1cm)},y={(-1cm,0)},every node/.style={scale=1},baseline=(current  bounding  box.center)]
\def \strx {0.25};
\def \stry {0.3};
%objects lvl1
\node (DD1) at (-1*\strx,3*\stry) {$D$};
\node (II1) at (1*\strx,3*\stry) {$I$};
%objects lvl3
\node (DD2) at (-1*\strx,-3*\stry) {$AD$};
\node (II2) at (1*\strx,-3*\stry) {$A$};
%objects lvl2
%arrowsMu
\path[->,font=\tiny,>=angle 90]
(DD1) edge node[above] {$d$} (DD2)
(II1) edge node[above] {$\eta$} (II2);
\end{tikzpicture}
\end{equation}
is a Hopf ring in $\mathcal{W}$. Hence, by Proposition \ref{prop:semi}, there is a semidirect product $H\rtimes A=DA\oplus A$, with the Hopf ring structure in $\mathcal{V}$ having components 
\begin{equation} \label{eq:monCross}
\begin{tikzpicture}[x={(0,-1cm)},y={(-1cm,0)},baseline=(current  bounding  box.center)]
\def \strx {0.2};
\def \stry {1};
%objects lvl1
\node (DD1) at (-3*\strx,3*\stry) {$DADA$};
\node (DI1) at (-1*\strx,3*\stry) {$DAIA$};
\node (ID1) at (1*\strx,3*\stry) {$IADA$};
\node (II1) at (3*\strx,3*\stry) {$IAIA$};
%
%objects lvl2
\node (D) at (-1*\strx, 0*\stry) {$DA$};
\node (I) at (1*\strx, 0*\stry) {$IA$};
%
%objects lvl3
\node (I1) at (0*\strx, -3*\stry) {$I\,\,$};
%
%arrowsMu
\path[->,font=\tiny,>=angle 90]
(DI1) edge node[fill=white,rounded corners=2pt,inner sep=1pt, above] {$1\mu$} (D)
(ID1) edge node[fill=white,rounded corners=2pt,inner sep=1pt, below] {$1\mu \circ s_{AD} 1$} (D)
(II1) edge node[below] {$\mu$} (I);
%arrowsEta
\path[->,font=\tiny,>=angle 90]
(I1) edge node[above] {$\eta$} (I);
\end{tikzpicture}
\end{equation}
\begin{equation} \label{eq:comonCross}
\begin{tikzpicture}[x={(0,-1cm)},y={(-1cm,0)},baseline=(current  bounding  box.center)]
\def \strx {0.2};
\def \stry {1};
%objects lvl1
\node (DD1) at (-3*\strx,3*\stry) {$DADA$};
\node (DI1) at (-1*\strx,3*\stry) {$DAIA$};
\node (ID1) at (1*\strx,3*\stry) {$IADA$};
\node (II1) at (3*\strx,3*\stry) {$IAIA$};
%
%objects lvl2
\node (D) at (-1*\strx, 0*\stry) {$DA$};
\node (I) at (1*\strx, 0*\stry) {$IA$};
%
%objects lvl3
\node (I1) at (0*\strx, -3*\stry) {$I\,.$};
%
%arrowsMu
\path[<-,font=\tiny,>=angle 90]
(DI1) edge node[fill=white,rounded corners=2pt,inner sep=1pt, above] {$1\delta$} (D)
(ID1) edge node[fill=white,rounded corners=2pt,inner sep=1pt, below] {$\tau_{D}1\circ 1\delta$} (D)
(II1) edge node[fill=white,rounded corners=2pt,inner sep=1pt, below] {$\delta$} (I);
%arrowsEta
\path[<-,font=\tiny,>=angle 90]
(I1) edge node[above] {$\epsilon$} (I);
\end{tikzpicture}
\end{equation}

\subsubsection{When $\mathcal{W}=\text{GAb}$}

Take $\mathcal{V}=\mathrm{Ab}$, and $A=Z(=\mathbb{Z}[x,x^{-1}])$. An $A$-comodule $(D,d)$ can be thought of as a graded abelian group (see Example \ref{ex:GAb}) with grades $D_i$. Condition (\ref{eq:Dcondition}) gives that for all $i,j\in\mathbb{Z}$ and $x\in D_i$ and $y\in D_j$
\begin{equation}\label{eq:ZDBraidingCondition}
(-1)^{ij}y\otimes x=-x\otimes y\,.
\end{equation}
The left hand side of the equality is an element of $D_j\otimes D_i$ component in the sum defining $(D\otimes D)_{i+j}$, while the right hand side is an element of $D_i\otimes D_j$. An argument similar to the one for Lemma \ref{lemma:0tensor} forces
\begin{align}
D_i\otimes D_j &=0,\text{ for }i\neq j\label{eq:ten0}\\
\sigma_{D_iD_i}&=(-1)^{i+1} 1_{D_iD_i}\,.
\end{align}
%since any non-zero element (in the tensor product) of the form
%\begin{equation}
%\sum_{l=1}^n x_l\otimes y_l
%\end{equation}
%would give that, for some $k$, elements $x_k$ and $y_k$ violate (\ref{eq:ZDBraidingCondition}).

Now, Lemma \ref{lemma:0tensor} further constrains the decomposition of individual $D_i$. We will consider those groups involved in  decomposition of either finitely generated groups, namely $\mathbb{Z}$ and $\mathbb{Z}/p_s^{n_s}\mathbb{Z}$, or divisible groups \cite{Kaplansky1954}, namely $\mathbb{Q}$ and
Pr\" ufer groups $\mathbb{Q}(p_s)$, where $p_s$ is the $s^\text{th}$ prime, and $n_s\in \mathbb{N}$. The tensor ``multiplication table'', up to isomorphism, for these groups ($\mathbb{Z}$ is omitted) is given by
\begin{equation}
\arraycolsep=1.1pt\def\arraystretch{1.2}
\begin{array}{|c|c|c|c|}
\hline
\otimes&
\mathbb{Z}/p_s^{n_s}\mathbb{Z}&
\mathbb{Q}(p_s)&
\mathbb{Q}\\
\hline
\mathbb{Z}/p_t^{n_t}\mathbb{Z}&
\delta_{s,t}\mathbb{Z}/p_s^{\min(n_s,n_t)}\mathbb{Z}&
0&
0\\
\hline
\mathbb{Q}(p_t)&
0&
0&
0\\
\hline
\mathbb{Q}&
0&
0&
\mathbb{Q}\\
\hline
\end{array} 
\end{equation}
Each of these components has $\sigma_{XX}=1_{XX}$ which can be shown using the following pattern
\begin{align}\label{eq:commArg}
u\otimes v = (me) \otimes (ne)
	= (ne) \otimes (me)
	= v \otimes u\,
\end{align}
and the fact that a pair $(u,v)$ of elements determines $e$ such that $u=me$ and $v=ne$ for some integers $m$ and $n$. In addition, when $X=\mathbb{Z}/2\mathbb{Z}\,,$ we have $\sigma_{XX}=-1\,.$ All this means that we can have:
\begin{itemize}
\item either only one copy of $\mathbb{Z}$ in one of the odd degrees
\item or no copies of $\mathbb{Z}$, at most one copy of $\mathbb{Z}/p_s^{n_s}\mathbb{Z}$ for each $s$ and fixed $n_s$ in odd degrees (except for $\mathbb{Z}/2\mathbb{Z}$ which can appear in an even degree) and arbitrary many copies of Pr\" ufer groups, at arbitrary degrees.
\end{itemize}

In the torsion-free-non-divisible part, we could ask for the following sufficient condition, slightly generalising\footnote{The first author is not aware of any examples of this that are not covered above.} the argument followed in (\ref{eq:commArg}).
\begin{lemma}
If each two elements $u$ and $v$ of an abelian group $G$ determine a set of elements $E={(e_i)}_{i\in I}$ such that $e_s\otimes e_t=0$ for $s\neq t$ and both $u$ and $v$ can be expressed as finite linear combination of elements from $E$, then $\sigma_{GG}=1_{GG}$.
\end{lemma}

\subsubsection{The Pareigis example}

Let $s=\pm1$ denote the degree of the differential,  $D_{s}=\mathbb{Z}$, and denote its generator by $d$. By the argument above, $D_{i\neq s}=0.$  The biring $(D\oplus \mathbb{Z})\rtimes \mathbb{Z}[x,x^{-1}]$ has underlying abelian group $Q:=\mathbb{Z}[x,x^{-1}]\oplus \mathbb{Z}\otimes \mathbb{Z}[x,x^{-1}]$. The (co)unit and (co)multiplication are determined using (\ref{eq:monCross}) and (\ref{eq:comonCross}):
\begin{eqnarray}
\begin{aligned}
Q \xrightarrow{\epsilon} \mathbb{Z} \ \text{; } \ d\otimes x^j \mapsto 0 \ \text{, } \  x^j \mapsto 1
\end{aligned}
\end{eqnarray}
\begin{eqnarray}
\begin{aligned}
\mathbb{Z}\xrightarrow{\eta} Q \text{; } \ 1\mapsto x^0
\end{aligned}
\end{eqnarray}
\begin{eqnarray}
\begin{aligned}
Q\xrightarrow{\delta} Q\otimes Q \text{; } \  
 x^k \mapsto x^k\otimes x^k   \text{, } \ \\
 d\otimes x^j\mapsto d\otimes x^j\otimes x^j+x^{s+j}\otimes d\otimes x^j 
\end{aligned}
\end{eqnarray}
\begin{eqnarray}
\begin{aligned}
Q\otimes Q \xrightarrow{\mu} Q \text{; } \  
d\otimes x^j\otimes d\otimes x^k \mapsto 0  \text{, } \ 
d\otimes x^j\otimes x^k \mapsto d\otimes x^{j+k} \text{, } \ \\
x^j\otimes d\otimes x^k \mapsto (-1)^j d\otimes x^{j+k}  \text{, } 
x^j\otimes x^k \mapsto x^{j+k} \phantom{AAAAA}
\end{aligned}
\end{eqnarray}

To see what the antipode is, consider the general antipode diagram (\ref{ap}), and label the edges
\begin{equation}
\begin{tikzpicture}[every node/.style={scale=0.6},baseline=(current  bounding  box.center)]
\def \strx {1}
\def \stry {0.4}
\def \angmnd {0}
\def \angio {0}

\coordinate (inH) at (-1 * \strx, 6 * \stry);
\coordinate (inA) at (0.75 * \strx, 6 * \stry);

\coordinate (apH) at (-1.5 * \strx, 0 * \stry);
\fill (apH) circle (1.5pt);
\coordinate (apA1) at (0.75 * \strx, 5.2 * \stry);
\fill (apA1) circle (1pt);

\coordinate (cmndA1) at (0.75 * \strx, 4.5 * \stry);
\coordinate (elt) at (0.5 * \strx, 2.5 * \stry);
\coordinate (apA2) at (0.8 * \strx, 1.75 * \stry);
\fill (apA2) circle (1pt);
\coordinate (cmndA2) at (0.8 * \strx, 1 * \stry);
\coordinate (mndA1) at (-0 * \strx, -1 * \stry);
\coordinate (mndA2) at (-0.5 * \strx, -2 * \stry);
\coordinate (mndA3) at (1 * \strx, -5 * \stry);

\coordinate (act) at (-1 * \strx, -5 * \stry);

\coordinate (outH) at (-1 * \strx, -6 * \stry);
\coordinate (outA) at (1 * \strx, -6 * \stry);

\node at (outH) [left] {$d^k_i$};
\node at (apH) [below left] {$d^k_i$};
\node at (inH) [left] {$(-1)^kd^k_i$};

\node at (outA) [left] {$x^j$};
\node at (mndA2) [above] {$x^i$};
\node at (elt) [above] {$(-1)^{-i(i+j)}$};
\node at (inA) [right] {$x^{-(i+j)}$};

\draw[thick] (inH) to[out=-90+\angio,in=90] (apH)
				   to[out=-90,in=\angmnd] (act)
				   to[out=-90,in=90-\angio] (outH);
\draw[very thin] (inA) to[out=-90-\angio,in=90] (apA1)
					   to[out=-90,in=90] (cmndA1);
\draw[very thin] (cmndA1) to[out=180+\angmnd,in=180-\angmnd] (mndA2);
\draw[very thin] (cmndA1) to[out=-\angmnd,in=180-\angmnd] (mndA3)
						  to[out=-90,in=90+\angio] (outA);
\draw[very thin] (elt) to[out=180,in=\angmnd] (mndA1);
\draw[very thin] (elt) to[out=0,in=90](apA2)
					   to[out=-90,in=90] (cmndA2);
\draw[very thin] (cmndA2) to[out=180+\angmnd,in=180-\angmnd] (mndA1);
\draw[very thin] (cmndA2) to[out=-\angmnd,in=\angmnd] (mndA3);
\draw[very thin] (mndA1) to[out=-90,in=\angmnd] (mndA2)
						 to[out=-90,in=180-\angmnd] (act);
\end{tikzpicture}
\end{equation}
where either $k=1$ and $i=s$, or $i=k=0$. So we have
\begin{eqnarray}
\begin{aligned}
Q \xrightarrow{s} Q\text{; } \
d\otimes x^j\mapsto (-1)^j d\otimes x^{-j-s}\text{, } \
x^j\mapsto x^{-j} \ .
\end{aligned}
\end{eqnarray}
We get exactly the Pareigis Hopf ring $P$ \eqref{PareigisEx} by setting
\begin{eqnarray}
\begin{aligned}
s=-1\text{, } \xi=x \text{, } \ \psi=d\otimes x^0 \ .
\end{aligned}
\end{eqnarray}

On the other hand, by setting
\begin{eqnarray}
\begin{aligned}
s=1\text{, } \xi=x \text{, } \ \psi=d\otimes x^0 
\end{aligned}
\end{eqnarray}
we get a modified Pareigis ring,
\begin{align}
P_+&=\mathbb{Z}\langle {\xi,\xi^{-1},\psi} \rangle/(\xi\psi+\psi\xi,\psi^2)\\
\Delta(\xi)&= \xi\otimes \xi,\; 
\epsilon(\xi)=1,\; s(\xi)=\xi^{-1} \nonumber \\
\Delta(\psi)&= \psi\otimes 1+\xi\otimes \psi,\;
\epsilon(\psi)=0,\; s(\psi)= \psi\xi^{-1}\,, \nonumber
\end{align}
which corresponds to exchanging $\xi$ and $\xi^{-1}$, and gives, as comodules, chains like (\ref{eq:diffArrows}).

\subsubsection{When $\mathcal{W}=\text{DGAb}$}

We can iterate the process by taking $\mathcal{W}=\text{DGAb}$, noting that the graded abelian group $D_s$, with $s=\pm 1$ has a unique differential structure, and that $d^{(D_s\otimes B)}=-d^{(B)},$ to obtain a second differential $d'$ on the chain $B$, satisfying
\begin{equation*}
\begin{tikzpicture}[xscale=1, yscale=0.66, every node/.style={scale=1},baseline=(current  bounding  box.center)]
%objects lvl1
\node (DD1) at (-1,1) {$B_{i}$};
\node (II1) at (1,1) {$B_{i-s}$};
%objects lvl3
\node (DD2) at (-1,-1) {$B_{i-1}$};
\node (II2) at (1,-1) {$B_{i-1-s}$};
%
%\node (tC) at (0,0) {\scriptsize $-1$};
%objects lvl2
%arrowsMu
\path[->,font=\tiny,>=angle 90]
(DD1) edge node[left] {$d$} (DD2)
(DD1) edge node[above] {$d'$} (II1)
(II1) edge node[right] {$-d$} (II2)
(DD2) edge node[below] {$d'$} (II2);
\end{tikzpicture}
\end{equation*}

\subsubsection{When $\mathcal{V}=\text{DGAb}$}

Now we consider the grading monoid in $\mathcal{V}=\text{DGAb}$. Note that $\text{DGAb}=\text{Ab-Cat}(\mathcal{D},\text{Ab})$ for a particular abelian category $\mathcal{D}$. Using the universal property of EM-objects in $\text{Ab-Cat}$ we can conclude that $Z-$comodules in DGAb are equivalently chains in GAb. So comodules $\beta:B\rightarrow Z\otimes B$ turn each component $B_n$ of the chain into a graded abelian group with components $B_{nm}$, and the differential $d_n:B_n\rightarrow B_{n-1}$ separates into components $d_{nm}:B_{nm}\rightarrow B_{n-1,m}$. Comodule morphisms $f:B\rightarrow C$ are chain maps respecting the grading; that is $f$ consists of components $f_{nm}:B_{nm}\rightarrow C_{nm}$ satisfying $f\circ d=d\circ f.$

Let $\kappa=\pm 1$ parametrise the two different coelements $Z$ can have: substitute $(-1)^{ij}$ in (\ref{diag:coel}) with $\kappa^{ij}$. Then, the tensor product
\begin{align}
(B\otimes C)_{nm} &= \sum_{i+j=n} (B_i\otimes C_j)_m\\
	&= \sum_{i+j=n}\sum_{p+q=m} B_{ip}\otimes C_{jq}\\
d(b\otimes c) &= db\otimes c + (-1)^i b\otimes dc
\end{align}
has braiding
\begin{equation}
\sigma(b\otimes c)=(-1)^{ij} \kappa^{pq} c\otimes b\,.
\end{equation}

For a differential $D$ we can choose a graded chain with $D_{s\frac{1+\kappa}{2},1}=\mathbb{Z}$ and $0$ for the other components (when $\kappa=1$, $s=\pm 1$ gives two different directions, as in the previous example). Note that $\sigma_{DD}=-1$ in all cases. So we can consider $(I\oplus D)\text{-}$comodules to obtain a second differential:
\begin{itemize}
\item for $\kappa=-1$ we get chains in DGAb, aka double complexes, with the second differential $d'_{nm}:B_{nm}\rightarrow B_{n,m-1}$ satisfying
\begin{equation*}
\begin{tikzpicture}[xscale=1, yscale=0.66, every node/.style={scale=1},baseline=(current  bounding  box.center)]
%objects lvl1
\node (DD1) at (-1,1) {$B_{n,m}$};
\node (II1) at (1,1) {$B_{n,m-1}$};
%objects lvl3
\node (DD2) at (-1,-1) {$B_{n-1,m}$};
\node (II2) at (1,-1) {$B_{n-1,m-1}$};
%
%\node (tC) at (0,0) {\scriptsize $-1$};
%objects lvl2
%arrowsMu
\path[->,font=\tiny,>=angle 90]
(DD1) edge node[left] {$d$} (DD2)
(DD1) edge node[above] {$d'$} (II1)
(II1) edge node[right] {$d$} (II2)
(DD2) edge node[below] {$d'$} (II2);
\end{tikzpicture}
\end{equation*}
\item for $\kappa=1$ we get a second differential $d'_{nm}:B_{nm}\rightarrow B_{n-1,m-1}$ satisfying
\begin{equation*}
\begin{tikzpicture}[xscale=1.2, yscale=0.66, every node/.style={scale=1},baseline=(current  bounding  box.center)]
%objects lvl1
\node (DD1) at (-1,1) {$B_{n,m}$};
\node (II1) at (1,1) {$B_{n-s,m-1}$};
%objects lvl3
\node (DD2) at (-1,-1) {$B_{n-1,m}$};
\node (II2) at (1,-1) {$B_{n-1-s,m-1}$};
%
%\node (tC) at (0,0) {\scriptsize $-1$};
%objects lvl2
%arrowsMu
\path[->,font=\tiny,>=angle 90]
(DD1) edge node[left] {$d$} (DD2)
(DD1) edge node[above] {$d'$} (II1)
(II1) edge node[right] {$-d$} (II2)
(DD2) edge node[below] {$d'$} (II2);
\end{tikzpicture}
\end{equation*}
\end{itemize}

\bibliographystyle{acm}
%\bibliography{../../Notes/references_c.bib}

\end{document}